\documentclass{article}
\usepackage[utf8]{inputenc}

\usepackage{latexsym,amsfonts,amssymb,amsmath,amsthm}

\usepackage{cite}

\usepackage{booktabs} 
\usepackage{caption} 
\usepackage{subcaption} 
\usepackage{graphicx}
\usepackage{graphics}
\usepackage{pgfplots}
\usepackage[all]{nowidow}
\usepackage[utf8]{inputenc}
\usepackage{tikz}
\usetikzlibrary{automata}
\usepackage{multicol,comment}
\usepackage{algpseudocode,algorithm,algorithmicx}
  \usepackage[frozencache=true,cachedir=minted-cache]{minted} 
\usepackage[inline]{enumitem} 
\definecolor{blue}{HTML}{1F77B4}
\definecolor{orange}{HTML}{FF7F0E}
\definecolor{green}{HTML}{2CA02C}

\pgfplotsset{compat=1.14}

\newtheorem{tm}{Theorem}[section]
\newtheorem{prop}[tm]{Proposition}

\newtheorem{lem}[tm]{Lemma}

\newtheorem{rk}[tm]{Remark}

\numberwithin{equation}{section}
\numberwithin{tm}{section}

\title{ Multiplicity of endemic equilibria for a diffusive SIS epidemic model with mass-action }
\author{Keoni Castellano\footnote{castek1@unlv.nevada.edu} \quad  and \quad  Rachidi B. Salako \footnote{rachidi.salako@unlv.edu}  
\\
\\
{\small Department of  Mathematical Sciences,  University of Nevada Las Vegas,}\\
{\small Las Vegas, NV 89154, USA}}
\date{}

\begin{document}

\maketitle

\begin{abstract} 

We study a diffusive   SIS epidemic model with the mass-action transmission mechanism and show, under appropriate assumptions on the parameters, the existence of multiple endemic equilibria (EE).  Our results answer some   open questions  on previous studies related to   disease extinction  or persistence when $\mathcal{R}_0<1$ and the multiplicity of EE solutions when $\mathcal{R}_0>1$.  Interestingly, even with such a  simple nonlinearity induced by the mass-action,  we show that the diffusive epidemic model may have an S-shaped or backward bifurcation curve of EE solutions.

\end{abstract}

\smallskip


\section{Introduction}

 \quad 
 As of May 2023, according to the World Health Organization (WHO)  coronavirus (COVID-19) dashboard, SARS-CoV-2 has infected more than seven hundred million  people  and   claimed more than six million deaths worldwide. This fast spread of the SARS-CoV-2 virus is partly due to globalization that has  made the world more connected. To  predict the dynamics of infectious diseases and  develop effective control strategies, researchers have developed and investigated epidemic models.

  \quad Consider   the ODE  Susceptible-Infected-Susceptible (ODE-SIS) mathematical epidemic system 

 \begin{equation}\label{model-001}
     \begin{cases}\frac{dS}{dt}= \gamma I -f(S,I)S & t>0,\cr 
     \frac{dI}{dt}= f(S,I)S-\gamma I & t>0,
     \end{cases}
 \end{equation}
 where $f$ is a locally Lipschitz  function on $\mathbb{R}_+^2$ satisfying $f(S,0)=0$ for $S\ge 0$. The  model \eqref{model-001}  describes the dynamics of a population living on a single patch and affected by  an  infectious disease. In  system \eqref{model-001}, $S$ and $I$ denote the size of the susceptible and infected populations, respectively; $\gamma$ is the disease recovery rate;  the function $f(S,I)$ accounts for the force of infection.  Two such functions are common: $f(S,I)=\beta\frac{I}{I+S}$ known as the frequency-dependent transmission mechanism, and $f(S,I)=\beta I$ called the mass-action transmission mechanism. Here, $\beta$ is the disease transmission rate. A basic fact about system \eqref{model-001} is that the total population size $N=S+I$ is constant. 
 Moreover, for simple force of infection functions  (e.g.,   frequency-dependent or  mass-action incidence mechanism) the {\it basic reproduction number } (BRN) 
 defined as  the expected number of new cases directly generated by one case in a population where all individuals are susceptible to infection, is enough to completely understand the dynamics of the disease:  the disease eventually extincts if BRN is less or equal to one, whereas it persists and eventually stabilizes at the unique endemic steady state solution if  BRN exceeds one. 
  Note that the ODE-SIS \eqref{model-001} assumes that the population is uniformly distributed and do not disperse.  In reality,  people move around  for several reasons including academic, economic, social, and political reasons, among others. This population movement contributes to the circulation of infectious diseases. Hence, 
   to obtain more precise predictions on infectious disease dynamics,  the ODE-SIS  model \eqref{model-001} must be adjusted to incorporate  population movement and spatial heterogeneity of the environment.

\quad In  2008, to study the effect of population movement and environmental heterogeneity on disease persistence, Allen et al. \cite{Allen2008}  included diffusion of populations into the system \eqref{model-001} with the frequency-dependent transmission mechanism and studied   the diffusive epidemic model 
\begin{equation}\label{e1-prime}
    \begin{cases}
    S_t=d_S\Delta S+\gamma I -\beta \frac{SI}{S+I} & x\in\Omega,\ t>0,\cr
    I_{t}=d_I\Delta I+\beta \frac{SI}{S+I}-\gamma I& x\in\Omega,\ t>0,\cr
     0=\partial_{\vec{n}}S=\partial_{\vec{n}}I & x\in\partial\Omega,\ t>0,\cr
    N=\int_{\Omega}(S+I),
    \end{cases}
\end{equation}
where $\Omega$ is a bounded domain in $\mathbb{R}^n$ ($n\ge 1$) with a smooth boundary $\partial\Omega$ and  $\vec{n}$ denotes the outward unit normal vector at $\partial\Omega$. $S(x,t)$ and $I(x,t)$ are the local densities of the susceptible and infected populations, respectively. Hence, $S$ and $I$ are both location dependent. The positive constants $d_S$ and $d_I$  represent the diffusion rates of the susceptible and infected populations, respectively, while $\beta$ and $\gamma$ are positive and H\"older continuous functions on $\overline{\Omega}$.  The authors of \cite{Allen2008} gave a variational formula for the BRN of \eqref{e1-prime}, denoted as $\mathcal{R}_1$  (see formula \ref{R-star-eq} below). They established that the disease will eventually be eradicated if $\mathcal{R}_1<1$, whereas system  \eqref{e1-prime} has a unique EE solution  if $\mathcal{R}_1>1$. Hence, as in  the corresponding ODE-SIS model \eqref{model-001}, the diffusive epidemic model \eqref{e1-prime} has a (unique) EE if and only if $\mathcal{R}_1>1$. For more studies on system \eqref{e1-prime} we refer to \cite{Allen2008,Peng2009a,Peng_Yi2013,Peng2009b}.  

 \quad Inspired by the above mentioned works, several studies have been devoted to  diffusive epidemic models (see \cite{Cui_Lou2016, 
 LP2022, LSS2023, LouSalako2021, Peng_Shi2008, DeJong1995, GKLZ2015, LS2023_2, Peng_Zhao, Salako2023_1, TW2023} and the references therein). In particular, Deng and Wu \cite{DengWu2016}  considered the diffusive counterpart of the ODE-SIS epidemic model \eqref{model-001} with the mass-action: 

\begin{equation}\label{e1}
    \begin{cases}
    S_t=d_S\Delta S+\gamma I -\beta SI & x\in\Omega,\ t>0,\cr
    I_{t}=d_I\Delta I+\beta SI-\gamma I& x\in\Omega,\ t>0,\cr
     0=\partial_{\vec{n}}S=\partial_{\vec{n}}I & x\in\partial\Omega,\ t>0,\cr
    N=\int_{\Omega}(S+I).
    \end{cases}
\end{equation}
The variables in \eqref{e1} have the same meanings as those in \eqref{e1-prime}.   They also found a variational formula for the BRN of \eqref{e1}, denoted as $\mathcal{R}_0$ (see \eqref{R-0-def} below) and established the existence of EE for $\mathcal{R}_0>1$. Furthermore, they obtained some partial results on the nonexistence of EE when $\mathcal{R}_0$ is sufficiently small, and uniqueness of EE when $\mathcal{R}_0$ is large. The works \cite{CastellanoSalako2021, Wu_Zou2016, Wen2018} further studied the asymptotic profiles of EE as the diffusion rates of the population get small or large. However, the following question remains. Does system \eqref{e1} have multiple  EE solutions for some range of  $\mathcal{R}_0$?   

 \quad 
\quad Our goal in this  study is to investigate the structure (i.e. multiplicity/uniqueness) of the EE solutions  of the diffusive  model \eqref{e1} and their asymptotic profiles for small $d_S$. Our results provide affirmative answer to the above mentioned open question. Indeed, roughly speaking, we show that there exist some  critical numbers $0<N_{\rm critic,1}\le N_{\rm critic,2}$, uniquely determined by the diffusion rate  of the  infected individuals $d_I$, such that system \eqref{e1} has no EE solution if  $N\le N_{\rm critic,1}$. However, if $N>N_{\rm critic,1}$ and  $d_S$ is small, then \eqref{e1} has a maximal EE solution.  In addition, if $N_{\rm critic,1}<N_{\rm critic,2}$ and $N_{\rm critic,1}<N<N_{\rm critic,2}$, it also has a minimal EE solution. This corresponds to the scenario of multiplicity of EE when $\mathcal{R}_0<1$ (see Theorem \ref{T0}). Meanwhile, we show that multiplicity of EE can never occur for large values of $d_S$ (see Theorem \ref{T1}). On the other hand, when $d_S$ is small, we show that  the multiplicity of EE solution may result either from a: (i) backward bifurcation at $N=N_{\rm critic,2}$ (see Theorem \ref{T2-2}), or (ii)  forward and S-shaped bifurcation at $N=N_{\rm critic,2}$ (see Theorem \ref{T2}). Moreover, in case (ii), we establish that system \eqref{e1} has at least three EE solutions for some range of $N>N_{\rm critic,2}$ and sufficiently small values of $d_S$, which corresponds to the scenario of multiplicity of the EE solution when $\mathcal{R}_0>1$. Furthermore, as a consequence of  the multiplicity of EE solutions, depending on how the movement rate of the susceptible individuals is lowered, the disease may either persist or die out (see Theorem \ref{T4}).     


\quad In Section \ref{main-results}, we state our main results. Section \ref{Preliminaries} contains some preliminaries  essential for the proofs of our main results, which are presented in Section \ref{proofs}. In the Appendix, we construct a few examples of parameters for which the hypotheses of our theorems hold.

\section{Main Results}\label{main-results}

\quad We state our main results in the first two subsections and compliment them with a conclusion. First, we introduce some notation. Given $q\in[1,\infty)$ and an integer $k\ge 1$, let $L^q(\Omega)$ denote the Banach space of $L^q$-integrable functions on $\Omega$, and  $W^{k,q}(\Omega)$ the usual Sobolev space. For every $d_I>0$, define 
\begin{equation}\label{R-star-eq}
    \mathcal{R}_1=\mathcal{R}_1(d_I)=\sup_{\varphi\in W^{1,2}(\Omega)\setminus\{0\}}\frac{\int_{\Omega}\beta\varphi^2}{\int_{\Omega}[d_I|\nabla \varphi|^2+\gamma\varphi^2]}.
\end{equation}
It is well known (see \cite{Allen2008}) that the supremum in \eqref{R-star-eq} is achieved and there is a unique positive function $\varphi_{1}\in C^2(\overline{\Omega})$ with $\|\varphi_{1}\|_{L^2(\Omega)}=1$ satisfying 
\begin{equation}\label{R-star-pde}
    \begin{cases}
    0=d_I\Delta \varphi-\gamma\varphi+\frac{1}{\mathcal{R}_1}\beta\varphi & x\in\Omega,\cr 
    0=\partial_{\vec{n}}\varphi & x\in\partial\Omega.
    \end{cases}
\end{equation}
Furthermore, any solution of \eqref{R-star-pde} is spanned by $\varphi_{1}$. Thanks to \cite{Allen2008}, the quantity $\mathcal{R}_1$ is the basic reproduction number for \eqref{e1-prime}. Moreover, when the  assumption {\bf (A)},

\medskip

\noindent{\bf (A)} The function $\frac{\beta}{\gamma}$ is not constant,

\medskip

holds, it follows from Lemma \ref{lem1} that $\mathcal{R}_1$ is strictly decreasing in $d_I$, and has an inverse function, which we denote by $\mathcal{R}_1^{-1}$. Throughout the remainder of this work, we shall always suppose that assumption {\bf (A)} holds. It follows from \cite{DengWu2016} that  the quantity  $\mathcal{R}_0=\mathcal{R}_0(N,d_I)$, defined by 
\begin{equation}\label{R-0-def}
    \mathcal{R}_0=\frac{N}{|\Omega|}\mathcal{R}_1,
\end{equation}
is the basic reproduction number of \eqref{e1}. It is clear from \eqref{R-0-def} that $\mathcal{R}_0$ is strictly increasing with respect to $N$ and independent of $d_S$. Moreover, if $d_I$, $\beta$ and  $\gamma$ are fixed, we can vary $\mathcal{R}_0$ from zero to infinity. In the statements of our results, the former parameters will be fixed while $\mathcal{R}_0$ would be often allowed to vary from zero to infinity.  
For convenience, given $h\in C(\overline{\Omega})$, we set $\overline{h}:=\int_{\Omega}h/|\Omega|$, $h_{\min}:=\min_{x\in\overline{\Omega}}h(x)$ and $h_{\max}:=\max_{x\in\overline{\Omega}}h(x)$. 

\subsection{Multiplicity/uniqueness of EE solutions of system \eqref{e1}}

\quad An equilibrium solution of \eqref{e1}  is a time independent solution. An equilibrium solution of the 
form $(S,0)$ is called a disease free equilibrium (DFE). Note that $(\frac{N}{|\Omega|},0)$  is the unique DFE of \eqref{e1}.  
 An equilibrium solution for which $I>0$  is called an endemic equilibrium (EE). Our first result reads as follows.

\begin{tm}[Multiplicity of EE]\label{T0}
 Fix $d_I>0$. There exists  $\mathcal{R}_0^{\rm low}=\mathcal{R}_0^{\rm low}(d_I)$  satisfying $0<\mathcal{R}_0^{\rm low}\le \min\big\{ 1,\overline{({\gamma}/{\beta})}\mathcal{R}_1\big\}$ such that  the following conclusions  hold.
\begin{itemize}
    \item[\rm (i)] If  $\mathcal{R}_0\le \mathcal{R}_0^{\rm low}$,  \eqref{e1}  has no EE solution for every $d_S>0$. 
    
    \item[\rm (ii)] If  $\mathcal{R}_0> \mathcal{R}_0^{\rm low}$, then  there is $d^{*}_1=d^*_1(\mathcal{R}_0,d_I)>0$  such that \eqref{e1} has an EE solution $(S_{\rm high},I_{\rm high})$ for every $0<d_S<d^*_1$. Furthermore,  any other EE solution $(S,I)$ of \eqref{e1}, if exists, must satisfy
    \begin{equation}\label{T0-eq1}
       I(x)<I_{\rm high}(x) \quad \forall\ x\in\overline{\Omega}. 
    \end{equation}

    \item[\rm (iii)]If $ \mathcal{R}_0^{\rm low}<1$ and  $\mathcal{R}^{\rm low}_{0}<\mathcal{R}_0<1$, then system \eqref{e1} has an   EE solution $(S_{\rm low},I_{\rm low})$  for every $0<d_S<d^{*}_1$, where $d_1^*$ is as in {\rm (ii)}, satisfying 
    \begin{equation}\label{T0-eq2}
        I_{\rm low}(x)<I_{\rm high}(x) \quad \forall\ x\in\overline{\Omega},
    \end{equation}
    such that any other EE solution $(S,I)$ of \eqref{e1}, if exists, must satisfy
    \begin{equation}\label{T0-eq3}
        I_{\rm low}(x)<I(x)\quad \forall\ x\in\overline{\Omega}. 
    \end{equation}
  \end{itemize}  
    
\end{tm}

Let $d_I>0$ and $ \mathcal{R}_0^{\rm low} $ be given by Theorem \ref{T1}. It follows from Theorem \ref{T0}-{\rm (i)} and {\rm (ii)} that the quantity  $\mathcal{R}_0^{\rm low}$ is a sharp critical number that the BRN must exceed for the existence of  EE of system \eqref{e1} for some range of the diffusion rate of susceptible population and total population size. The asymptotic profiles of the EE solutions of Theorem \ref{T0} as $d_{S}\to 0$ will be given in Theorem \ref{T4}.  An important quest is to know when $\mathcal{R}_0^{\rm low}<1$. In this direction, we have: 

\begin{prop}\label{prop0_1} Suppose that \begin{equation}\label{lem2-eq3}
       \overline{(\gamma/\beta)}<\overline{\gamma}\big/\overline{\beta}.
    \end{equation} 
    Then $\mathcal{R}_0^{\rm low}<1$ for every   $d_I>\mathcal{R}^{-1}_1\big(\overline{(\gamma/\beta)}\big)$. 
\end{prop}
Note that \eqref{lem2-eq3}  holds when $\gamma=\beta^2$ and is not constant. Under hypothesis \eqref{lem2-eq3}, we see that there is a range of parameters satisfying $\mathcal{R}_0<1$ such that \eqref{e1} has at least two EE solutions for small  $d_S$. An immediate question is to know whether \eqref{e1} may have multiple EE solutions for large values of $d_S$. 

\begin{tm}[Uniqueness of EE]\label{T1} For every $d_I>0$, there exists $d_{\rm low}=d_{\rm low}(d_I)$ satisfying $0\le d_{\rm low}<d_I$ such that the following hold.
\begin{itemize}
    \item[\rm (i)]  If $d_S>d_{\rm low}$ and $\mathcal{R}_0>1$, then system \eqref{e1} has a unique EE solution.
    \item[\rm (ii)] If $d_S>d_{\rm low}$ and $\mathcal{R}_0\le 1$, then  system \eqref{e1} has no EE solution.
\end{itemize}     
\end{tm}

By Theorem \ref{T1}, $\mathcal{R}_0$ is enough to predict the existence  of EE solution of \eqref{e1} for large values of $d_S$. Note that $d_{\rm low}(d_I)$ is independent of $N$ and strictly less than $d_I$. Hence for $d_{\rm low}(d_I)<d_S<d_I$, system \eqref{e1} has a (unique) EE solution if and only if $\mathcal{R}_0>1$. This improves previously known results on the uniqueness of EE solutions of \eqref{e1} (see \cite{DengWu2016} where uniqueness is obtained for $d_S\ge d_I$). 
\begin{rk}\label{RK0_1} We note that 
if $d_{\rm low}(d_I)>0$, then  for every $0<d_S<d_{\rm low}(d_I)$, system \eqref{e1} has at least two EE solutions for a range of $\mathcal{R}_0$. (The proof of this statement will be given in Section \ref{proofs} right after the proof of Theorem \ref{T1}.) This shows that $d_{\rm low}(d_I)$ is a sharp critical number that $d_S$ must exceed to guarantee the uniqueness of EE solutions of \eqref{e1}.  
    
\end{rk}

\quad Our next result complements Theorem \ref{T0} by identifying sufficient conditions that lead to a backward bifurcation curve of EE solutions at $\mathcal{R}_0=1$.

\begin{tm}[Backward bifurcation curve]\label{T2-2}Fix $d_I>0$ and suppose that   \begin{equation}\label{T2-2-eq1}
\overline{\beta\varphi_{1}^3}<(\overline{\varphi_{1}})(\overline{\beta\varphi_{1}^2}).
\end{equation} 
There is $d^*_2=d_2^*(d_I)>0$ such that for every $0<d_S<d_2^*$,  as $\mathcal{R}_0$ increases from zero to infinity, the EE solutions of system \eqref{e1} form an unbounded simple connected curve which bifurcates from the left at $\mathcal{R}_0=1$.
    
\end{tm}

Proposition \ref{appen-prop3}-{\rm (ii)} of the Appendix gives an example of parameters satisfying \eqref{T2-2-eq1}.  Our next result also complements Theorem \ref{T0} by identifying sufficient conditions that lead to a forward S-shaped bifurcation curve of EE solutions at $\mathcal{R}_0=1$.

\begin{tm}[Forward and S-shaped bifurcation curve]\label{T2} Fix $d_I>0$ and suppose that 
\begin{equation}\label{T2-eq1}
\overline{(\gamma/\beta)}\mathcal{R}_1<1\quad \text{and}\quad  \overline{\beta\varphi_{1}^3}>(\overline{\varphi_{1}})(\overline{\beta\varphi_{1}^2}).
\end{equation}
Then there is ${d}_3^*={d}_3^*(d_I)>0$ such that for every $0<d_S<{d}_3^*$, as $\mathcal{R}_0$ increases from zero to infinity, the EE solutions of system \eqref{e1} form an unbounded simple connected curve which bifurcates from the right at $\mathcal{R}_0=1$. Furthermore, for every $0<d_S<{d}_3^*$, there exist $\mathcal{R}_{0,1}^{d_S}<1<\mathcal{R}_{0,2}^{d_S}\le \mathcal{R}_{0,3}^{d_S}$ such that:
\begin{itemize}
    \item[\rm (i)] if  $\mathcal{R}_0<\mathcal{R}_{0,1}^{d_S}$, system \eqref{e1} has no EE solution;
    \item[\rm (ii)] if $\mathcal{R}_0=\mathcal{R}_{0,1}^{d_S}$, system \eqref{e1} has at least one EE solution;
     \item[\rm (iii)] if  either $\mathcal{R}_{0,1}^{d_S}<\mathcal{R}_0\le1$ or $\mathcal{R}_0=\mathcal{R}_{0,2}^{d_S}$, system \eqref{e1} has at least two EE solutions;
    \item[\rm (iv)]   if $1<\mathcal{R}_0<\mathcal{R}_{0,2}^{d_S}$, system \eqref{e1} has at least three EE solutions.
    \item[\rm (v)] if $\mathcal{R}_0>\mathcal{R}_{0,2}^{d_S}$, system \eqref{e1} has at least one EE solution,  which is unique if $\mathcal{R}_0>\mathcal{R}_{0,3}^{d_S}$.
\end{itemize}
Moreover,  $\mathcal{R}^{d_S}_{0,i}$ is strictly increasing in $d_S$ for each $i=1,2,3$; and as $d_S\to 0$,  $\mathcal{R}_{0,1}^{d_S}\to \mathcal{R}_0^{\rm low}$ and  $\mathcal{R}_{0,i}^{d_S}\to \mathcal{R}_{0,i}^*$, $i=2,3$, for some positive numbers $1<\mathcal{R}^{*}_{0,2}\le \mathcal{R}^*_{0,3}$.
\end{tm}

\begin{figure}[htb]
    \centering
    \includegraphics[scale=0.25]{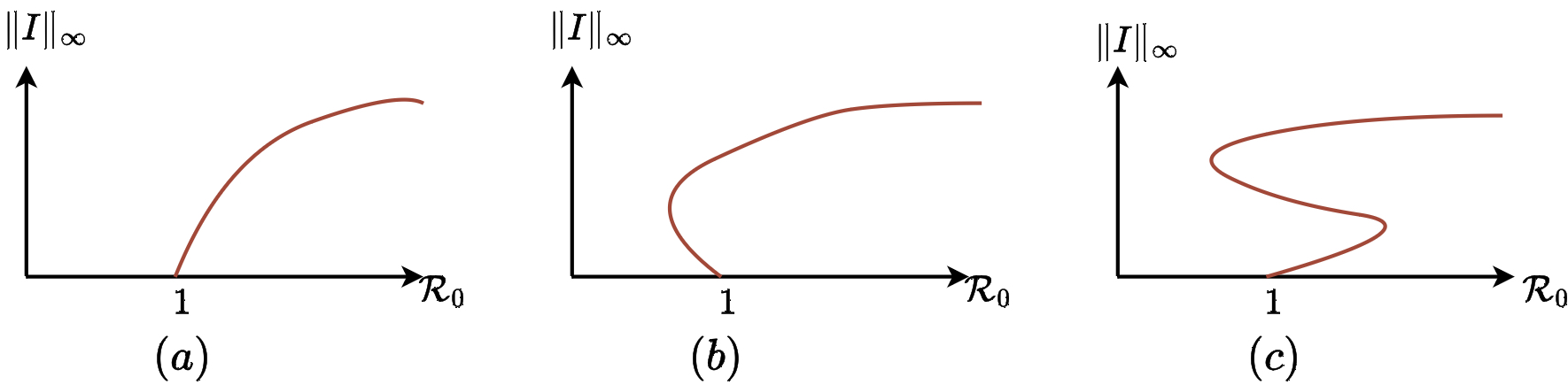}
    \caption{Schematic bifurcation curves of EE solutions of system \eqref{e1}. Figure \ref{fig1_1}-$(a)$ corresponds to large values of $d_S$ as indicated by Theorem \ref{T1}. Figure \ref{fig1_1}-$(b)$ corresponds to small values of $d_S$ and a backward bifurcation curve of EE as indicated by Theorem \ref{T2-2}. Figure \ref{fig1_1}-$(c)$ corresponds to small values of $d_S$ and a forward and S-shape bifurcation curve of EE as indicated by Theorem \ref{T2}. }
    \label{fig1_1}
\end{figure}

Proposition \ref{appen-prop3}-{\rm (i)} of the Appendix gives an example of parameters for which hypothesis \eqref{T2-eq1} holds.

\subsection{Asymptotic profiles of the EE solution for small $d_S$.}
Next, we explore how the multiplicity of the EE solutions of \eqref{e1} may affect disease control strategy. To this end,   we complement Theorems \ref{T0}, \ref{T2-2} and \ref{T2} with a result on the asymptotic profiles for the EE  as $d_S\to 0$. Our result indicate that the disease may either persist or die out depending on how $d_S$ is lowered. Precisely, we have the following result. 

\begin{tm}\label{T4} Fix $d_I>0$ and suppose that $\mathcal{R}_0^{\rm low}<1$, where $\mathcal{R}_0^{\rm low}$ is given by Theorem \ref{T0}. \begin{itemize}
\item[\rm (i)]Fix $\mathcal{R}_0^{\rm low}<\mathcal{R}_0<1$ and let $d_1^*$ be as in Theorem \ref{T0}-{\rm (ii),(iii)} such that system \eqref{e1} has a maximal EE solution $(S_{\rm high},I_{\rm high})$ and a minimal EE solution $(S_{\rm low},I_{\rm low})$ for every $0<d_S<d_1^*$. 
   
    \begin{itemize}
        \item[\rm (i-1)] If  $\mathcal{R}_0<\overline{({\gamma}/{\beta})}\mathcal{R}_1$, then there is a positive constant $C>0$ such that \begin{equation}\label{T0-eq4}
        \frac{d_{S}}{C}\le I_{\rm low}< I_{\rm high}\le Cd_{S}\quad  \forall\ 0<d_S<\frac{d_1^*}{2}\quad \text{and}
    \end{equation}
    \begin{equation}\label{T0-eq5}
        S_{\rm high}\to S_{\rm high}^*:=\frac{N(1-d_Iu^*_{\rm high})}{\int_{\Omega}(1-d_Iu^*_{\rm high})} \quad \text{and}\quad  S_{\rm low}\to S_{\rm low}^*:=\frac{N(1-d_Iu^*_{\rm low})}{\int_{\Omega}(1-d_Iu^*_{\rm low})}\quad \text{as}\ d_S\to 0
    \end{equation}
    in $C^{1}(\overline{\Omega})$ where $ 0<u^*_{\rm low}<u^*_{\rm high}<\frac{1}{d_I}$ are  classical solutions of the nonlocal elliptic equation
    \begin{equation}\label{T0-eq6}
    \begin{cases}
    0=d_I\Delta u^*+\big( (\mathcal{R}_0/\mathcal{R}_1)\beta\frac{(1-d_Iu^*)}{\overline{(1-d_Iu^*)}}-\gamma\big)u^* & x\in\Omega,\cr
    0=\partial_{n}u^* & x\in\partial\Omega.
    \end{cases}
    \end{equation}

        \item[\rm (i-2)] If  $\mathcal{R}_0>\overline{({\gamma}/{\beta})}\mathcal{R}_1$, then 
        \begin{equation}\label{T0-eq7}
        \lim_{d_S\to0^+}\left[\big\|S_{\rm high}-{\gamma}/{\beta}\big\|_{\infty}+\big\|I_{\rm high}-\big({\mathcal{R}_0}/{\mathcal{R}_1}-\overline{({\gamma}/{\beta})}\big)\big\|_{\infty}\right] = 0
    \end{equation}
    and $(S_{\rm low},I_{\rm low})$ satisfies \eqref{T0-eq4} and \eqref{T0-eq5}.
    \end{itemize} 
    \item[\rm(ii)] In addition, suppose that \eqref{T2-eq1} holds. Let $d_3^*$ and $\mathcal{R}_{0,2}^*$ be given by Theorem \ref{T2} and  fix $1<\mathcal{R}_0<\mathcal{R}^*_{0,2}$.  Then for every $0<d_S<d_3^*$, system \eqref{e1} has a maximal EE solution $(S_{\rm high},I_{\rm high})$ and a minimal EE solution $(S_{\rm low},I_{\rm low})$. Moreover, as $d_S\to 0$, $I_{\rm low}$ satisfies \eqref{T0-eq4}, $S_{\rm low}$ satisfies \eqref{T0-eq5} up to a subsequence, and  $(S_{\rm high},I_{\rm high})$ satisfies \eqref{T0-eq7}. 
    \end{itemize}
\end{tm}
Theorem \ref{T4}-{\rm (i-2)} $\&$ {\rm (ii)} show that as $d_S\to 0$, the profiles of the EE solutions of \eqref{e1} depend on the chosen subsequence. These results also show that the two scenarios discussed in \cite[Theorem 2.3 ]{Wu_Zou2016} are possible. Theorem \ref{T4} also complements \cite[Theorem 2.5]{CastellanoSalako2021} which establishes the asymptotic profiles of EE of \eqref{e1} as $d_I\to 0$.

\subsection*{Conclusion.} We examined the questions of multiplicity or uniqueness of   the EE solutions of a diffusive epidemic model with the mass-action transmission and obtained some interesting  results. In particular, our results revealed some new phenomona, which cannot be observed from  neither the ODE-model \eqref{model-001} nor from the corresponding PDE-model with the frequency-dependent transmission \eqref{e1-prime}. 

\quad As mentioned above,  for the ODE-SIS model \eqref{model-001} with simple nonlinearity (i.e.,   frequency-dependent, mass-action transmission), the BRN  is enough to completely characterize the existence  of EE solutions. 
This is  also the case  for the diffusive epidemic model \eqref{e1-prime} with  the frequency-dependent transmission.     However, for the diffusive model \eqref{e1} with the mass-action transmission, we showed that the BRN is not enough to predict the persistence of the disease. 

\quad Indeed,  Theorem \ref{T0} indicates that, for the dynamics of solutions of  model \eqref{e1}, the disease may still persist even if  $\mathcal{R}_0<1$. Precisely, there is some critical number  $0<\mathcal{R}_0^{\rm low}\le 1$, uniquely determined by $d_I$, such that \eqref{e1} has no EE if $\mathcal{R}_0\le \mathcal{R}_0^{\rm low}$ for any diffusion rate $d_S$ of the susceptible host. However, thanks to Theorem \ref{T0}-{\rm (iii)}, if $\mathcal{R}_0^{\rm low}<\mathcal{R}_0<1$, then there exist at least two EE solutions for sufficiently small values of $d_S$. In this this case, we see that $\mathcal{R}_0$ is not enough to predict the persistence of the disease. In Proposition \ref{prop0_1} we showed  that if the average of the ratio of the recovery rate over the transmission rate is smaller than the ratio of the average of the recovery rate over the average of the transmission rate, then $\mathcal{R}_0^{\rm low}<1$ for large values of $d_I$.

\quad Second, by  Theorem \ref{T2-2}, if  the  average of $\beta\varphi_{1}^3$ is smaller than the product of the  averages of $\varphi_{1}$ and $\beta\varphi_{1}^2$, again $\mathcal{R}_0^{\rm low}<1$. In fact, in this case, there is a backward bifurcation at $\mathcal{R}_0=1$. See Figure \ref{fig1_1}-{(b)} for a schematic diagram of the backward bifurcation curve of EE solutions.

\quad Third, when $\mathcal{R}_0>1$, Theorem \ref{T2} shows that system \eqref{e1} may have at least three EE solutions for small $d_S$. See Figure \ref{fig1_1}-{(c)} for a schematic diagram of the S-shape bifurcation curve of EE solutions of \eqref{e1} at $\mathcal{R}_0=1$. Despite all the above possible interesting scenarios that could exhibit the global structure of the EE solutions of system \eqref{e1}, Theorem \ref{T1} shows that $\mathcal{R}_0$ is enough to predict the persistence of the disease if  the susceptible population moves slightly slower or faster than the infected population. See Figure \ref{fig1_1}-{(a)} for a schematic picture of  bifurcation curve of the EE solutions of \eqref{e1} if $d_S>d_{\rm low}(d_I)$, where $d_{\rm low}(d_I)<d_I$ is given by Theorem \ref{T1}. 

\quad 
Thanks to the above results and Theorem \ref{T4}, we can conclude that a decrease in movement rate of susceptible individuals could be an effective disease control strategy if accumulation of population  is also minimized. 

%

\section{Preliminary Results}\label{Preliminaries}
\quad We present a few preliminary results here.  We start with the following result from \cite{Allen2008}.

\begin{lem}\label{lem1} 
{\rm (i)} If $\frac{\beta}{\gamma}$ is constant, then $\mathcal{R}_1=\frac{\beta}{\gamma}$ for all $d_I>0$. 
{\rm (ii)} If  $\frac{\beta}{\gamma}$ is not constant, then $\mathcal{R}_1$ is strictly decreasing in $d_I$,
    \begin{equation}\label{R-star-limit-eq}
        \lim_{d_I\to 0^+}\mathcal{R}_1=\max_{x\in\overline{\Omega}}({\beta}/{\gamma)(x)}\quad \text{and}\quad \lim_{d_I\to\infty}\mathcal{R}_1=\overline{\beta}/\overline{\gamma}.
    \end{equation} 
    In particular, if \eqref{lem2-eq3} holds, then
      $   1<\overline{({\gamma}/{\beta})}\mathcal{R}_1$  {if} $ 0<d_I<\mathcal{R}_1^{-1}\big(1/\overline{({\gamma}/{\beta})}\big)$; $ 
        1=\overline{({\gamma}/{\beta})}\mathcal{R}_1$ if $  d_I=\mathcal{R}^{-1}_1\big(1/\overline{({\gamma}/{\beta})}\big)$; and $ 
        1>\overline{({\gamma}/{\beta})}\mathcal{R}_1$  if $d_I>\mathcal{R}_1^{-1}\big(1/\overline{({\gamma}/{\beta})}\big)$ .
    
\end{lem}

{\quad  For each $q\in[2,\infty)$, consider the linear  elliptic operator  on $L^{q}(\Omega)$ (resp. on $C(\overline{\Omega})$), defined by
$$
\mathcal{L}^*w=d_I\Delta w+(l^*\beta-\gamma)w \quad w\in {\rm Dom}_{q}, (\text{resp.}\ w\in {\rm Dom}_{\infty} )\quad , 
$$ 
where ${\rm Dom}_{q}=\{u\in W^{2,q}(\Omega) | \partial_{\vec{n}}u=0\ \text{on}\ \partial\Omega\}$ and  $
{\rm Dom}_{\infty}=\left\{u\in\cap_{q\ge 1}{\rm Dom}_{q}| \ \mathcal{L}^* u\in C(\overline\Omega)\right\}.$  By \cite[Theorem 1, page 357]{Evans}, since $\mathcal{L}^*$ is symmetric on $L^2(\Omega)$, then ${L}^2(\Omega)$ has an orthonormal basis formed by eigenfunctions of $\mathcal{L}^*$. Moreover,  since by \eqref{R-star-pde} equation $\varphi_{1}$ is an eigenvector of $\mathcal{L}^*$ associated with its principal eigenvalue zero, we can decompose $L^2(\Omega)$ as $L^2(\Omega)={\rm span}(\varphi_1)\oplus\{w\in L^2(\Omega)| \int_{\Omega}w\varphi_{1}=0\}$.  Thus,   
$$
C(\overline{\Omega})={\rm span}(\varphi_{1})\oplus\mathcal{Z}_{\infty}\quad \text{and}\quad L^q(\Omega)={\rm span}(\varphi_{1})\oplus \mathcal{Z}_q,\quad q\ge 2, 
$$
where $\mathcal{Z}_{\infty}=\{g\in C(\overline{\Omega})|\int_{\Omega}g\varphi_{1}=0 \}$ and $\mathcal{Z}_q=\{w\in L^q(\Omega)| \int_{\Omega}w\varphi_{1}=0\}$ for $q\ge 2$. Furthermore, we have that $\mathcal{L}^*_{|\mathcal{Z}_q\cap{\rm Dom}_{q}}\ :\ \mathcal{Z}_q\cap{\rm Dom}_{q} \to \mathcal{Z}_q$ is invertible for each $q\in[2,\infty]$.

\quad \quad Next, consider the one-parameter family of elliptic equations introduced in \cite{CastellanoSalako2021} :
\begin{equation}\label{Eq1}
    \begin{cases}
    0=d_I\Delta u +(l\beta(1-d_Iu)-\gamma)u & x\in\Omega,\cr 
    0=\partial_{\vec{n}}u & x\in\partial\Omega.
    \end{cases}
\end{equation}
Note that \eqref{Eq1} is a one-parameter family of the classical diffusive logistic elliptic equation. Hence, several interesting results have been established as with respect to the existence, uniqueness of stability of its positive solution whenever it exists.  Throughout the rest of the paper, we set $l^*:=\frac{1}{\mathcal{R}_1}$, so that $\mathcal{R}_0=N/(l^*|\Omega|)$. The next lemma collects some results on the positive solution of \eqref{Eq1}.

\begin{lem}\label{lem2} Fix $d_I>0$ and let $\mathcal{R}_1$ be given by \eqref{R-star-eq}.
\begin{itemize}
    \item[\rm (i)] The elliptic equation \eqref{Eq1} has a (unique) positive solution, $u^l$,  if and only if $l>l^*$. Moreover, 
    \begin{equation}\label{lem2-eq1}
        0<u^l<\frac{1}{d_I},\quad \lim_{l\to l^*}\|u^l\|_{\infty}=0,\quad \lim_{l\to\infty}\|u^l-\frac{1}{d_I}\|_{\infty}=0,\quad  \lim_{l\to\infty}\Big\|l(1-d_Iu^l)-\frac{\gamma}{\beta}\Big\|_{\infty}=0,
    \end{equation}
    {and \begin{equation}\label{rev-0}
      \lim_{l\to l^*}\Big\|\frac{u^l}{l-l^*}-\Big(\frac{\mathcal{R}_1\int_{\Omega}\beta\varphi_1^2}{d_I\int_{\Omega}\beta\varphi_1^3}\Big)\varphi_1\Big\|_{C^1(\overline{\Omega})}=0.
  \end{equation} }

    \item[\rm (ii)] The mapping $\big(l^*,\infty\big)\ni l\mapsto u^l\in C^1(\overline{\Omega})$ is smooth and strictly increasing. Setting $v^l=\partial_l u^l$ for every $l>l^*$, then $v^l$ satisfies
    \begin{equation}\label{v-l-eq1}
        \begin{cases}
        0=d_I\Delta v^l+(l\beta(1-2d_Iu^l)-\gamma)v^l+\beta(1-d_Iu^l)u^l & x\in\Omega,\cr 
        0=\partial_{\vec{n}}v^l & x\in\partial\Omega,
        \end{cases}
    \end{equation}
    \begin{equation}\label{v-l-eq2}
       \lim_{l\to l^*}\left\|v^l-\left(\frac{\mathcal{R}_1\int_{\Omega}\beta\varphi_1^2}{d_I\int_{\Omega}\beta\varphi_1^3}\right)\varphi_1\right\|_{C^1(\overline{\Omega})}=0,\quad \text{and}\quad \lim_{l\to\infty}\Big\|l^2v^l-\frac{\gamma}{d_I\beta}\Big\|_{\infty}=0.
    \end{equation}

    \item[\rm (iii)] The function $\mathcal{N}_{d_I}$ defined by 
    \begin{equation}\label{N-d_I-def}
        \mathcal{N}_{d_I}(l)
        =\frac{l}{l^*|\Omega|}\int_{\Omega}{(1-d_Iu^l)}=l\overline{(1-d_Iu^l)}\mathcal{R}_1\quad \forall\ l\ge l^*,
    \end{equation}
    is continuously differentiable. Furthermore, \begin{equation}\label{lem2-eq2}
        \mathcal{N}_{d_I}(l^*)
        =1\quad \text{and}\quad \lim_{l\to\infty}\mathcal{N}_{d_I}(l)
        ={\overline{(\gamma/\beta)}}\mathcal{R}_1.
    \end{equation} 
\end{itemize}

\end{lem}
\begin{proof}{\rm (i)} It follows from standard results on the diffusive logistic equations. See for example \cite[Lemma 4.2-(i)]{CastellanoSalako2021}. { Note also that \eqref{lem2-eq1} is proved in \cite[Lemma 4.2-(i)]{CastellanoSalako2021}}. So, we shall show that \eqref{rev-0} holds.   To this end, we first write $u^l$  as
\begin{equation}\label{rev-1}
    u^{l}=(l-l^*)\psi^l\quad \forall\ l>l^*,
\end{equation}
and then find the limit of $\psi^l$ as $l$ approaches $l^*$. To achieve this,   we set 
\begin{equation}\label{rev-8}
c(l):=\int_{\Omega}\varphi_1\psi^{l}\quad \text{and}\quad \tilde{\psi}^l:=\psi^l-c(l)\varphi_1\quad \forall\ l> l^*.
\end{equation}
Observing that $\int_{\Omega}\tilde{\psi}^l\varphi_1=0$ since $\int_{\Omega}\varphi_1^2=1$, we have that 
\begin{equation}\label{rev-2}
  \quad\tilde{\psi}^l\in\mathcal{Z}_{\infty}\quad \text{and}\quad  u^l=(l-l^*)(c(l)\varphi_1+\tilde{\psi}^l)\quad \forall\ l>l^*.
\end{equation}
Next, setting $\hat{\psi}^l :=\frac{\tilde{\psi}^l}{l-l^*}$, then 
\begin{equation}\label{rev-2-21}
  \quad\hat{\psi}^l\in\mathcal{Z}_{\infty}\quad \text{and}\quad  u^l=(l-l^*)(c(l)\varphi_1+(l-l^*)\hat{\psi}^l)\quad \forall\ l>l^*.
\end{equation}
{\bf Step 1.} In this step, we show that there is a positive constant $K_{d_I}>0$ such that
\begin{equation}\label{rev-6}
    \frac{|\Omega|\beta_{\min}\varphi_{1,\min}}{ld_I\|\beta\|_{\infty}K_{d_I}}\le c(l)\le \frac{|\Omega|\|\beta\|_{\infty}\|\varphi_{1}\|_{\infty}}{ld_I\beta_{\min}}
    \quad \forall\ l>l^*.
\end{equation}
Multiply \eqref{R-star-pde} and \eqref{Eq1} by $u^l$ and $\varphi_{1}$ respectively. After integrating the resulting equations and then taking their difference, we obtain 
$$
0=\int_{\Omega}(l\beta -ld_I\beta u^l -\gamma)u^l\varphi_{1}-\int_{\Omega}(l^*\beta-\gamma)u^{l}\varphi_1.
$$
This gives 
\begin{equation}\label{rev-10}
    (l-l^*)\int_{\Omega}\beta u^l\varphi_{1}=d_Il\int_{\Omega}\beta (u^l)^2\varphi_{1}\quad \forall\ l>l^*.
\end{equation}
 Using the H\"older's inequality and \eqref{rev-10}, for all $l>l^*$, we obtain that 
\begin{align*}
    (l-l^*)\int_{\Omega}u^l\varphi_{1}\ge \frac{ld_I}{\|\beta\|_{\infty}}\int_{\Omega}\Big(u^l\varphi_{1}\Big)^2\frac{\beta}{\varphi_1} 
    \ge \frac{ld_{I}}{\|\beta\|_{\infty}}\Big(\frac{\beta}{\varphi_1}\Big)_{\min}\int_{\Omega}\Big(u^{l}\varphi_{1}\Big)^2 
    \ge \frac{ld_I\beta_{\min}}{|\Omega|\|\beta\|_{\infty}\|\varphi_{1}\|_{\infty}}\Big(\int_{\Omega}u^l\varphi_{1}\Big)^2.
\end{align*}
As a result,
\begin{equation}\label{rev-7}
\frac{|\Omega|\|\beta\|_{\infty}\|\varphi_{1}\|_{\infty}}{ld_I\beta_{\min}}\ge\frac{1}{l-l^*}\int_{\Omega}u^l\varphi_{1}\quad \forall\ l>l^*.
\end{equation}
On the other hand, it follows from \eqref{rev-1} and \eqref{rev-8} that 
\begin{equation}\label{rev-11}
\int_{\Omega}u^{l}\varphi_{1}=(l-l^*)\int_{\Omega}\psi^l\varphi_{1}=c(l)(l-l^*) \quad \forall\ l>l^*.
\end{equation}
By \eqref{rev-11} and \eqref{rev-7}, we have that
\begin{equation}\label{rev-9}
    c(l)\le \frac{|\Omega|\|\beta\|_{\infty}\|\varphi_{1}\|_{\infty}}{ld_I\beta_{\min}}\quad \forall\ l>l^*.
\end{equation}
Next, since $\|l\beta(1-d_Iu^l)-\gamma\|_{\infty}\to 0$ as $l\to\infty$ and $\|l\beta(1-d_Iu^l)-\gamma\|_{\infty}\to \|l^*\beta-\gamma\|_{\infty}$ as $l\to l^*$, then 
$$
\sup_{l>l^*}\|l\beta(1-d_Iu^l)-\gamma\|_{\infty}<\infty.
$$
It then follows from the Harnack's inequality for elliptic equations and the fact that $u^l$ solves \eqref{Eq1} that there is a positive constant $K_{d_I}$ independent of $l>l^*$ such that 
\begin{equation*}
    \|u^l\|_{\infty}\le K_{d_I}u^l_{\min}\quad \forall\ l>l^*.
\end{equation*}
This together with \eqref{rev-10}  yields that, for every $l>l^*$, 
\begin{align*}
    (l-l^*)\int_{\Omega}u^l\varphi_{1}\le \frac{d_Il}{\beta_{\min}}\int_{\Omega}\beta(u^l)^2\varphi_{1} \le\frac{d_IlK_{d_I}\|\beta\|_{\infty}u^l_{\min}}{\beta_{\min}}\int_{\Omega}u^l\varphi_{1}\le\frac{ld_IK_{d_I}\|\beta\|_{\infty}}{|\Omega|\varphi_{1,\min}\beta_{\min}}\Big(\int_{\Omega}u^l\varphi_{1}\Big)^2.
\end{align*}
Hence, by \eqref{rev-11},
\begin{equation}\label{rev-12}
\frac{|\Omega|\beta_{\min}\varphi_{1,\min}}{ld_IK_{d_I}\|\beta\|_{\infty}}\le\frac{1}{l-l^*}\int_{\Omega}u^l\varphi_{1}=c(l)\quad \forall\ l>l^*.
\end{equation}
It is clear from \eqref{rev-9} and \eqref{rev-12} that \eqref{rev-6} holds.\\
{\bf Step 2.} In this step, we show that 
\begin{equation}
    \lim_{l\to l^*}c(l)=\frac{\int_{\Omega}\beta\varphi_{1}^2}{d_Il^*\int_{\Omega}\beta\varphi^3_{1}}.
\end{equation}

Note from \eqref{rev-2} and the fact that $u^l$ solves \eqref{Eq1} that 
\begin{equation}\label{rev-3}
    \begin{cases}
    0=d_I\Delta(c(l)\varphi_{1}+\tilde{\psi}^l) +(l\beta(1-d_Iu)-\gamma)(c(l)\varphi_{1}+\tilde{\psi}^l) & x\in\Omega,\cr 
    0=\partial_{\vec{n}}\tilde{\psi}^l & x\in\partial\Omega.
    \end{cases}
\end{equation}
Moreover, since $\varphi_{1}$ satisfies \eqref{R-star-pde}, we get from \eqref{rev-3} that 
\begin{equation}\label{rev-4}
    \begin{cases}
    0=d_I\Delta\tilde{\psi}^l+(l-l^*)c(l)\beta( 1-ld_I(c(l)\varphi_{1}+\tilde{\psi}^l)))\varphi_{1} +(l\beta(1-d_Iu)-\gamma)\tilde{\psi}^l & x\in\Omega,\cr 
    0=\partial_{\vec{n}}\tilde{\psi}^l & x\in\partial\Omega.
    \end{cases}
\end{equation}
Recalling that $\hat{\psi}^l=\frac{\tilde{\psi}^l}{l-l^*}$ for every $l>l^*$, we deduce from \eqref{rev-4} that 
\begin{equation}\label{rev-5}
    \begin{cases}    0=\mathcal{L}^*\hat{\psi}^l+c(l)\beta( 1-ld_I(c(l)\varphi_{1}+(l-l^*)\hat{\psi}^l)))\varphi_{1} +\beta(l(1-d_Iu)-l^*)\hat{\psi}^l & x\in\Omega,\cr 
    0=\partial_{\vec{n}}\hat{\psi}^l & x\in\partial\Omega,
    \end{cases}
\end{equation}
which gives that 
$$
\hat{\psi}^l=\mathcal{L}^{*,-1}_{|\mathcal{Z}_q\cap{\rm Dom}_q}\Big(c(l)\beta( 1-ld_I(c(l)\varphi_{1}+(l-l^*)\hat{\psi}^{l})))\varphi_{1} +\beta(l(1-d_Iu)-l^*)\hat{\psi}^l \Big)\quad \forall\ q\ge 2,\  l>l^*.
$$
 Setting $M_q:=\big\|\mathcal{L}^{*,-1}_{|\mathcal{Z}_q\cap{\rm Dom}_q}\big\|$, then
\begin{align}\label{YTYT1}
    \|\hat{\psi}^l\|_{W^{2,q}(\Omega)} 
    \le & M_q\|c(l)\beta( 1-ld_I(c(l)\varphi_{1}+(l-l^*)\hat{\psi}^{l})))\varphi_{1} +\beta(l(1-d_Iu)-l^*)\hat{\psi}^l\|_{L^q(\Omega)}\cr 
    \le & M_q\|\beta\|_{\infty}\Big((l-l^*)(ld_Ic(l)\|\varphi_{1}\|_{\infty}+1)+ld_I\| u^l\|_{\infty}\Big)\|\hat{\psi}^l\|_{W^{2,q}(\Omega)}\cr
    &+ M_qc(l)\|\beta(1-ld_Ic(l)\varphi_{1})\|_{\infty}|\Omega|^{\frac{1}{q}}.
\end{align}
Observing from \eqref{rev-6}
 and the fact that $u^l\to 0$ as $l\to l^*$ that, for each $q\ge 1$, 
 $$ M_q\|\beta\|_{\infty}\Big((l-l^*)(ld_Ic(l)\|\varphi_{1}\|_{\infty}+1)+ld_I\| u^l\|_{\infty}\Big)\|\to 0\quad \text{as}\ l\to l^*,
 $$
 then, for each $q\ge 2$, there is $ K^*_{q}>0$ such that 
 \begin{equation}\label{b-rev-1}
     \|\hat{\psi}^l\|_{W^{2,q}(\Omega)}\le K^*_q\quad \forall\ 0<l^*<l<l^*+1.
 \end{equation}
 Choosing $q\gg n$ such that $W^{2,q}(\Omega)$ is compactly embedded in $C^1(\overline{\Omega})$, it follows from \eqref{rev-6}, \eqref{b-rev-1}, and the fact that $u^l\to 0$ as $l\to l^*$ (possibly after passing to a subsequence)  that $c(l)\to c^*$  and $\hat{\psi}^l\to \hat{\psi}^*$ in $C^{1}(\overline{\Omega})$ as $l\to l^*$, where  $c^*$ is a positive number and $ \hat{\psi}^*\in C^2(\Omega)\cap\mathcal{Z}_{\infty}$  satisfies 
\begin{equation}\label{rev-13}
    \begin{cases}
        -\mathcal{L}^*\hat{\psi}^*=c^*\beta(1-l^*d_Ic^*\varphi_{1})\varphi_{1} & x\in\Omega,\cr 
        0=\partial_{\vec{n}}\hat{\psi}^* & x\in\partial\Omega.
    \end{cases}
\end{equation}
Hence, since $\mathcal{L}^*({\rm Dom}_{\infty}\cap \mathcal{Z}_{\infty})=\mathcal{Z}_{\infty}\subset\mathcal{Z}_2$,  multiplying \eqref{rev-13} by $\varphi_{1}$ and integrating the resulting equation on $\Omega$, we obtain that 
\begin{align*}
    0=c^*\int_{\Omega}\beta(1-l^*c^*d_I\varphi_{d_I})\varphi_{1}^2,
\end{align*}
which due to the fact that $c^*>0$ gives
$ 
c^*=\Big({\int_{\Omega}\beta\varphi_{1}^2}\Big)/\Big({l^*d_I\int_{\Omega}\beta\varphi_{1}^3}\Big).
$  Since $c^*$ is independent of the chosen subsequence, we then conclude that $c(l)\to c^*$ as $l\to l^*$. Moreover, since $\mathcal{L}^*$ is invertible on $\mathcal{Z}_{\infty}\cap {\rm Dom}_{\infty}$, we have that $\hat{\psi}^*$ is the unique solution of \eqref{rev-13} in $\mathcal{Z}_{\infty}\cap{\rm Dom}_{\infty}$ and  $\hat{\psi}^l\to \hat{\psi}^* $ in $C^1(\overline{\Omega})$ as $l\to l^*$. Therefore, we have 
$$
\frac{u^l}{l-l^*}=\psi^l=c(l)\varphi_{1}+(l-l^*)\hat{\psi}^l\to \left(\frac{\int_{\Omega}\beta\varphi_{1}^2}{l^*d_I\int_{\Omega}\beta\varphi_{1}^3}\right)\varphi_{1}\quad \text{as}\ l\to l^* \ \text{ in }\ C^1(\overline{\Omega}).
$$

\quad {\rm (ii)} The regularity of $u^l$ with respect to $l$ and the fact that $v^l>0$ and solves \eqref{v-l-eq1} is already proved in \cite[Lemma 4.2-(i)]{CastellanoSalako2021}.  Next, we prove that \eqref{v-l-eq2} holds.  For each $l>l^*$, let $\psi^l=\frac{u^l}{l-l^*}=c(l)\varphi_{1}+\tilde{\psi}^l$ be defined as in \eqref{rev-1}. Thanks to \eqref{rev-0}, for any positive number $A>0$, 
$$
1-d_Iu^l-lA\psi^l\to 1-Ac^*\varphi_{1}\quad \text{as}\quad l\to l^*\ \ \text{uniformly in}\ \Omega,
$$
where $c^*=\int_{\Omega}\beta\varphi_{1}^2/(d_Il^*\int_{\Omega}\beta\varphi_{1}^3)>0$. Therefore, we can choose $0<A_1<A_2$ such that 
\begin{equation*}
  1-d_Iu^l-lA_2\psi^l<0<1 -d_Iu^l-lA_1\psi^l\quad l^*<l<l^*+\varepsilon_0 
\end{equation*}
for some $\varepsilon_0>0$.  Hence, thanks to \eqref{rev-3}, we have that 
\begin{equation*}
    d_I\Delta (A_1\psi^l)+(l\beta(1-2d_Iu^l)-\gamma)(A_1\psi^l)+\beta(1-d_Iu^l)u^l= \beta(1-d_Iu^l-lA_1\psi^l)u^l>0 \quad x\in\Omega
\end{equation*}
and 
\begin{equation*}
    d_I\Delta (A_2\psi^l)+(l\beta(1-2d_Iu^l)-\gamma)(A_2\psi^l)+\beta(1-d_Iu^l)u^l= \beta(1-d_Iu^l-lA_2\psi^l)u^l<0 \quad x\in\Omega
\end{equation*}
for every $l^*<l<l^*+\varepsilon_0$. It then follows from \eqref{v-l-eq1} and the comparison principle for elliptic equations that 
\begin{equation*}
    A_1\psi^l<v^l<A_2\psi^l\quad 0<l^*<l<l^*+\varepsilon_0.
\end{equation*}
Therefore, by \eqref{v-l-eq1} and the estimates for elliptic equations (possibly after passing to a subsequence), there is a strictly positive function $v^*\in C^2(\Omega)$ and $v^l\to v^*$ as $l\to l^*$ in $C^1(\overline{\Omega})$. Moreover, $v^*$ satisfies
$$
\begin{cases}
    0=d_I\Delta v^*+(l^*\beta-\gamma)v^* & x\in\Omega\cr
    0=\partial_{\vec{n}}v^* & x\in\partial\Omega.
\end{cases}
$$
So, we must have that $v^*=c^{**}\varphi_{1}$ for some positive number $c^{**}$. Next,  we multiply \eqref{rev-3} and \eqref{v-l-eq1} by $v^l$ and $\psi^l$ respectively, and then integrate on $\Omega$. After taking the difference of the resulting equations and using $u^l=(l-l^*)\psi^l$, we obtain that 
$$0=\int_{\Omega}\beta(1-d_Iu^l-ld_Iv^l)(\psi^l)^2\quad \forall\ l>l^*.$$
 Letting $l\to l^*$ in the last equation yields 
$ 
0=\int_{\Omega}\beta(1-d_Il^*c^{**}\varphi_{1})(c^*\varphi_{1})^2.
$ Solving for $c^{**}$, we get  $ 
c^{**}=c^*
$  is independent of the chosen subsequence. Therefore,  $v^l\to c^*\varphi_{1}$ as $l\to l^*$ in $C^1(\overline{\Omega})$.
 
\quad Finally, set $p^l=l^2v^{l}$ for each $l>l^*$. By direct computations, it follows from \eqref{v-l-eq1} that $p^l$ satisfies
\begin{equation}\label{D1}
    \begin{cases}
        0=\frac{d_I}{l}\Delta p^l +\frac{\beta}{l}\big( z^l-\frac{\gamma}{\beta}\big)p^l+u^l\beta(z^l-d_Ip^l) & x\in\Omega,\cr 
        0=\partial_{\vec{n}}p^l & x\in\partial\Omega.
    \end{cases}
\end{equation}
where $z^l=l(1-d_Iu^l)$ for all $l>l^*$. Therefore, since $u^{l}\to \frac{1}{d_I}$ and $z^l\to\frac{\gamma}{\beta}$ as $l\to\infty$ in $C(\overline{\Omega})$ by \eqref{lem2-eq1}, we can employ the singular perturbation theory for elliptic equations to deduce from \eqref{D1} that $p^l\to \frac{\gamma}{d_I\beta}$ as $l\to\infty$ uniformly on $\overline{\Omega}$, which completes the proof of  \eqref{v-l-eq2}.

\quad {\rm (iii)} It follows from \cite[Lemma 4.2-(iii)]{CastellanoSalako2021}.
\end{proof}

The next result shows the relationship between EE solutions of \eqref{e1} and the solutions of \eqref{Eq1}

\begin{lem}\label{lem3} Let $d_S>0$, $d_I>0$ and $\mathcal{R}_0>0$.
\begin{itemize}
    \item[\rm (i)] If $(S,I)$ is an EE solution of \eqref{e1} then 
    \begin{equation}\label{kappa-def}
        \kappa =d_SS+d_II
    \end{equation} 
    is a constant function. Furthermore, setting 
    \begin{equation}\label{tilde-s-i-ded}
        \tilde{S}=\frac{S}{\kappa}\quad \text{and}\quad \tilde{I}=\frac{I}{\kappa},
    \end{equation}
    then $\tilde{I}=u^{\frac{\kappa}{d_S}}$ is the positive solution of \eqref{Eq1} with $l=\frac{\kappa}{d_S}>l^*$. Moreover, since $\mathcal{R}_0=N/(l^*|\Omega|)$, then $\mathcal{R}_0=\mathcal{N}_{d_I}(\kappa/d_S)+\frac{\kappa}{|\Omega|l^*}\int_{\Omega}u^{\frac{\kappa}{d_S}}$, where $\mathcal{N}_{d_I}$ is defined in \eqref{N-d_I-def}.
    \item[\rm (ii)] If for some $l>l^*$, $\mathcal{R}_0$ satisfies  
    \begin{equation}\label{N-equation}
 \mathcal{R}_0=\mathcal{N}_{d_I,d_S}(l):=\mathcal{N}_{d_I}(l)+\frac{ld_S}{l^*|\Omega|}\int_{\Omega}u^l,
    \end{equation}
    where $\mathcal{N}_{d_I}(l)$ is defined as in \eqref{N-d_I-def}, then $(S,I):=(l(1-d_Iu^l),d_Slu^l)$ is an EE solution of \eqref{e1}.
\end{itemize}
\end{lem}
\begin{proof} The proof is similar to that of \cite[Lemma 4.4]{CastellanoSalako2021}, hence it is omitted.
    
\end{proof}

\section{Proofs of Main Results}\label{proofs}

We present the proof of our main results in this section.

\begin{proof}[Proof of Theorem \ref{T0}] Let $d_I>0$ and define \begin{equation}\label{C0} 
 \mathcal{R}_0^{\rm low}=\inf_{l\ge l^*}\mathcal{N}_{d_I}(l),
\end{equation}
where $\mathcal{N}_{d_I}(l)$ is defined by \eqref{N-d_I-def}.  It is clear from \eqref{lem2-eq2} that $0\le \mathcal{R}_0^{\rm low}\le \{1,\overline{(\gamma/\beta)}\mathcal{R}_1\}$. Furthermore, since $\mathcal{N}_{d_I}(l)>0$ for every $l\ge l^*$ and converges to a positive number as $l$ approaches infinity, then $\mathcal{R}_0^{\rm low}>0$. 
Next,  we prove assertions  {\rm (i)}-{\rm (iii)}.

\quad {\rm (i)}   If $(S,I)$ is an EE solution of \eqref{e1} for some $d_S>0$, then by Lemma \ref{lem3}, $l:=\frac{\kappa}{d_S}>l^*$, where $\kappa$ is defined by \eqref{kappa-def}. Furthermore, by Lemma \ref{lem3}-{\rm (i)}, we have that 
$$ 
\mathcal{R}_0
=\mathcal{N}_{d_I}(\frac{\kappa}{d_S}) +\frac{\kappa}{l^*|\Omega|}\int_{\Omega}u^{\frac{\kappa}{d_S}}>\mathcal{R}_0^{\rm low}.
$$
This shows that \eqref{e1} has no EE solution whenever $\mathcal{R}_0\le \mathcal{R}_0^{\rm low} $ and $d_S>0$.

\quad {\rm (ii)} Suppose that $\mathcal{R}_0>\mathcal{R}_0^{\rm low}$. Therefore there is $l(\mathcal{R}_0,d_I)>l^*$ such that $\mathcal{N}_{d_I}(l(\mathcal{R}_0,d_I))<\mathcal{R}_0$.
   Set 
\begin{equation}\label{C4-2}
    d_1^*:=
        \frac{(\mathcal{R}_0-\mathcal{N}_{d_I}(l(\mathcal{R}_0,d_I)))|\Omega|l^*}{l(\mathcal{R}_0,d_I)\int_{\Omega}u^{l(\mathcal{R}_0,d_I)}}>0. 
\end{equation}
For every $d_S>0$, consider the function $\mathcal{N}_{d_I,d_S}$ defined as in \eqref{N-equation}. 
     Then $\mathcal{N}_{d_I,d_S}$ is continuously differentiable in $l\ge l^*$. Furthermore, 
    \begin{equation}\label{EE1}
\mathcal{N}_{d_I,d_S}(l^*)=1 \quad \text{and}\quad \lim_{l\to\infty}\mathcal{N}_{d_I,d_S}(l)=\infty.
    \end{equation} 
    Next, fix $0<d_S<d_1^*$. It follows from \eqref{C4-2} that  \begin{align}\label{C2}
    \mathcal{N}_{d_I,d_S}(l(\mathcal{R}_0,d_I))=&\mathcal{N}_{d_I}(l(\mathcal{R}_0,d_I))+\frac{d_Sl(\mathcal{R}_0,d_I)}{l^*|\Omega|}\int_{\Omega}u^{l(\mathcal{R}_0,d_I)}\cr <&\mathcal{N}_{d_I}(l(\mathcal{R}_0,d_I))+\frac{d(\mathcal{R}_0,d_I)l(\mathcal{R}_0,d_I)}{l^*|\Omega|}\int_{\Omega}u^{l(\mathcal{R}_0,d_I)} =\mathcal{R}_0.
    \end{align}
      Thus, by the intermediate value theorem, there is $l(\mathcal{R}_0,d_I,d_S)>l(\mathcal{R}_0,d_I)$ such that $\mathcal{N}_{d_I,d_S}(l(\mathcal{R}_0,d_I,d_S))=\mathcal{R}_0$. This together with \eqref{EE1} imply that the quantity  
      \begin{equation}\label{l-high-def}
        l_{\rm high}(d_S):=\max\{l>l(\mathcal{R}_0,d_I)\ :\ \mathcal{N}_{d_I,d_S}(l)=\mathcal{R}_0\}
    \end{equation}
    is a positive real number.  Observe that
    \begin{equation}\label{a-rev-0}
        \mathcal{N}_{d_I,d_S}(l_{\rm high}(d_S))=\mathcal{R}_0\quad \text{and}\quad \mathcal{N}_{d_I,d_S}(l)>\mathcal{R}_0\quad \forall\ l>l_{\rm high}(d_S).
    \end{equation}
     By Lemma \ref{lem3}-(ii),
    \begin{equation}\label{high-ee}(S_{\rm high},I_{\rm high}):=(l_{\rm high}(d_S)(1-d_Su^{l_{\rm high}(d_S)}),d_Sl_{\rm high}(d_S)u^{l_{\rm high}(d_S)})
    \end{equation}
    is an EE solution of  \eqref{e1}. Finally, we show that \eqref{T0-eq1} holds. So, suppose that $(S,I)$ is another EE solution of \eqref{e1}. Then, by Lemma \ref{lem3}  we  have that $\frac{\kappa}{d_S}>l^*$ and  $ I=d_{S}(\frac{\kappa}{d_S})u^{\frac{\kappa}{d_S}}$. Hence, since the mapping $(l^*,\infty)\ni l\mapsto lu^{l}$ is strictly increasing, and $\mathcal{N}_{d_I,d_S}(\frac{\kappa}{d_S})=\mathcal{R}_0=\mathcal{N}_{d_I,d_S}(l_{\rm high}(d_S))$, then  $\frac{\kappa}{d_S}<l_{\rm high}(d_S)$, which yields $I=d_{S}(\frac{\kappa}{d_S})u^{\frac{\kappa}{d_S}}<I=d_{S}l_{\rm high}(d_S)u^{l_{\rm high}(d_S)}=I_{\rm high}$. 

    \quad {\rm (iii)}  Suppose that $\mathcal{R}_0^{\rm low}<1$ and $\mathcal{R}_0^{\rm low}<\mathcal{R}_0<1$. 
     Let $d_1^*$ be given by {\rm (i)} and   $l(\mathcal{R}_0,d_I)$ be as in the proof of {\rm (ii)}. Fix $0<d_S<d_1^*$. Observe that 
    $  
    \mathcal{N}_{d_I,d_S}(l(\mathcal{R}_0,d_I))<\mathcal{R}_0<1=\mathcal{N}_{d_{I},d_S}(l^*)$.  Therefore, by the intermediate value theorem, there is $\tilde{l}(\mathcal{R}_0,d_I,d_S)\in (l^*,l(\mathcal{R}_0,d_I))$ such that $\mathcal{N}_{d_I,d_S}(\tilde{l}(\mathcal{R}_0,d_I,d_S))=\mathcal{R}_0$. This implies that the quantity

    \begin{equation}\label{l-low-def}
        l_{\rm low}(d_S):=\min\{l\in [l^*,l(\mathcal{R}_0,d_I))\ :\ \mathcal{N}_{d_I,d_S}(l)=\mathcal{R}_0\}
    \end{equation}
    is well defined and satisfies $l^*<l_{\rm low}(d_S)<l(\mathcal{R}_0,d_I)$.  Observe that
    \begin{equation}\label{a-rev-1}
        \mathcal{N}_{d_I,d_S}(l_{\rm low}(d_S))=\mathcal{R}_0\quad \text{and}\quad \mathcal{N}_{d_I,d_S}(l)>\mathcal{R}_0\quad \forall\ l^*\le l< l_{\rm low}(d_S)).
    \end{equation}
      Now, by Lemma \ref{lem3}-(ii),
    \begin{equation}\label{low-ee}(S_{\rm low},I_{\rm low}):=(l_{\rm low}(d_S)(1-d_Su^{l_{\rm low}(d_S)}),d_Sl_{\rm low}(d_S)u^{l_{\rm low}(d_S)})
    \end{equation}
    is an EE solution of  \eqref{e1}.  Since $l_{\rm low}(d_S)<l(\mathcal{R}_0,d_I)<l_{\rm high}(d_S)$ and the mapping $(l^*,\infty)\ni l\mapsto lu^l$ is strictly increasing, then \eqref{T0-eq2} holds.  Using again the fact that the mapping $(l^*,\infty)\ni l\mapsto lu^l$ is strictly increasing, it can be shown as in the case of \eqref{T0-eq1} that any other EE solution of \eqref{e1}, if exists, must satisfy \eqref{T0-eq3}.   
    \end{proof}

\begin{proof}[Proof of Proposition \ref{prop0_1}]
    It follows from Lemma \ref{lem1}-{\rm (ii)} that $1>
    \overline{(\gamma/\beta)}\mathcal{R}_1$ for every $d_I>\mathcal{R}_1^{-1}\big(1/\overline{(\gamma/\beta)}\big)$. Hence, $1>\overline{(\gamma/\beta)}\mathcal{R}_1\ge \mathcal{R}_0^{\rm low}$ for every $d_I>\mathcal{R}_1^{-1}\big(1/\overline{(\gamma/\beta)}\big)$.
\end{proof}

Next, we give a proof of Theorem \ref{T1}. 

\begin{proof}[Proof of Theorem \ref{T1}] Fix $d_I>0$ and define $M_{d_I}^*:=\frac{d_I}{|\Omega|}\sup_{l>l^*}\int_{\Omega}(u^l+lv^l)$, where for every $l>l^*$, $u^l$ and $v^l$ are the unique positive solutions of \eqref{Eq1} and \eqref{v-l-eq1}, respectively. Note from  \eqref{lem2-eq1} and \eqref{v-l-eq2} that $u^l\to \frac{1}{d_I}$ and $lv^l\to 0$ as $l\to\infty$ uniformly on $\Omega$.  Hence, 
$ \int_{\Omega}(u^l+lv^l)\to {|\Omega|}/{d_I}$ as $l\to\infty$. 
Note also from \eqref{lem2-eq1} and \eqref{v-l-eq2} that 
$ \int_{\Omega}(u^l+lv^l)\to {\Big(\int_{\Omega}\varphi_{1}\Big)\Big(\int_{\Omega}\beta\varphi_{1}^2\Big)}/{\big(d_I\int_{\Omega}\beta\varphi_{1}^3\big)}
$ as $l\to l^*$.  Hence 
$ \max\Big\{1,\frac{\big(\int_{\Omega}\varphi_{1}\big)\big(\int_{\Omega}\beta\varphi_{1}^2\Big)}{|\Omega|\int_{\Omega}\beta\varphi_{1}^3}\big\}\le M^*_{d_I}<\infty$.
Defining $    m_{d_I}^*:=1/{M^*_{d_I}}$, then $0<m^*_{d_I}\le 1$.  As a result, $d_{\rm low}=(1-m_{d_I}^*)d_I$ satisfies $0\le d_{\rm low}<d_I$.

\quad For every $d_S>0$, consider the function $\mathcal{N}_{d_I,d_S}$ be defined as in \eqref{N-equation}. Taking the derivative of the function $\mathcal{N}_{d_I,d_S}$  (see \eqref{N-equation}) with respect to $l$, we get 
\begin{equation}\label{C1-1}
    \frac{d\mathcal{N}_{d_I,d_S}(l)}{dl}=\frac{1}{l^*|\Omega|}\Big( |\Omega|+(d_S-d_I)\int_{\Omega}(lv^l+u^l)\Big) \quad \forall\ l>l^*.
\end{equation} 
From this point, we suppose that $d_S>d_{I}(1-m_{d_I}^*)$ and then show that 
\begin{equation}\label{C1-3}
\frac{d\mathcal{N}_{d_I,d_S}(l)}{dl}>0 \quad \forall\ l>l^*.
\end{equation}
If $d_S\ge d_I$, it is easy to see from \eqref{C1-1} that \eqref{C1-3} holds. So, we suppose that $d_{\rm low}<d_S<d_I$. Hence, by \eqref{C1-1}, we have that 
\begin{align*}
  \frac{d\mathcal{N}_{d_I,d_S}(l)}{dl} = \frac{1}{l^*}\left(1-\Big(1-\frac{d_S}{d_{I}}\Big)\frac{d_I}{|\Omega|}\int_{\Omega}(u^l+lv^l)\right) 
  \ge  \frac{M_{d_I}^*}{d_Il^*}\Big(d_S-d_{\rm low}\Big)>0,  
\end{align*}
which shows that \eqref{C1-3} holds in this subcase as well. Therefore, when $d_S>d_{\rm low}$, the mapping $\mathcal{N}_{d_I,d_S}$ is strictly increasing on $[l^*,\infty)$. Hence, thanks to Lemma \ref{lem3}, \eqref{e1} has an EE solution if and only if $\mathcal{R}_0>1$. Moreover, in this case, when an EE exists, it is unique. 
\end{proof}
\begin{proof}[Proof of Remark \ref{RK0_1}] Fix $d_I>0$ and let $M_{d_I}^*$ and $m_{d_I}^*$ be as in the proof of Theorem \ref{T1} so that $d_{\rm low}=d_I(1-m_{d_I}^*)$. Now  suppose that $0<d_{\rm low}$ and  fix $0<d_S<d_{\rm low}$. Hence,  $m_{d_I}^*<1-\frac{d_S}{d_I}$, which gives  
 $ 
0<{1}/({1-\frac{d_S}{d_I}})<{1}/{m_{d_I}^*}=M_{d_I}^*. $
Therefore, there is $l_0>l^*$ such that 
$  
0<\frac{1}{1-\frac{d_S}{d_I}}<\frac{d_I}{|\Omega|}\int_{\Omega}(u^{l_0}+l_0v^{l_0}),
$ 
 which implies that 
\begin{equation}\label{GH1}
    \frac{d\mathcal{N}_{d_I,d_S}(l_0)}{dl}=\frac{1}{l^*}\left(1-\Big(1-\frac{d_S}{d_{I}}\Big)\frac{d_I}{|\Omega|}\int_{\Omega}(u^{l_0}+l_0v^{l_0})\right)<0.
\end{equation}
On the other hand, we know from \eqref{lem2-eq1} and \eqref{v-l-eq2} that 
\begin{equation}\label{GH2}
    \lim_{l\to\infty}\frac{d\mathcal{N}_{d_I,d_S}(l)}{dl}=\frac{1}{l^*|\Omega|}\Big(|\Omega|+(d_S-d_I)\frac{|\Omega|}{d_I}\Big)=\frac{d_S}{d_Il^*}>0.
\end{equation}
Thanks to \eqref{GH1} and \eqref{GH2}, we deduce that  there are $l_0<l_1<l_2$ such that $\mathcal{N}_{d_I,d_S}(l_1)=\mathcal{N}_{d_I,d_S}(l_2)$. As a result, for $\mathcal{R}_0=\mathcal{N}_{d_I,d_S}(l_1)=\mathcal{N}_{d_I,d_S}(l_2)$, we have from Lemma \ref{lem3} that 
$$ 
(S_1,I_1)=(l_1(1-d_Iu^{l_1}),d_Sl_1u^{l_1})
\quad 
\text{and} 
\quad  
(S_2,I_2)=(l_2(1-d_Iu^{l_2}),d_Sl_2u^{l_2})
$$
 are two distinct EE solutions of \eqref{e1}. This completes the proof of the remark.
    
\end{proof}

Next, we give a proof of Theorem \ref{T2-2}. 

 \begin{proof}[Proof of Theorem \ref{T2-2}]  Suppose that \eqref{T2-2-eq1} holds.  By \eqref{v-l-eq2} and \eqref{C1-1}, we obtain  
 \begin{equation}\label{LK1}
 \frac{d\mathcal{N}_{d_I}(l^*)}{dl}= \Big(1-{\big(\overline{\varphi_1}\big)\big(\overline{\beta\varphi_1^2}\big)}\big/{\overline{\beta\varphi_1^3}}\Big)/l^*<0.
 \end{equation}
Thanks to \eqref{LK1}, we can chose $d_2^*>0$ small enough such that $ \frac{d\mathcal{N}_{d_I,d_S}(l^*)}{dl}<0$ for all $0\le d_S<d_2^*$. Now, fix $0<d_S<d_2^*$. 
Define the curve $\mathcal{F}_{d_I,d_S} : [l^*,\infty)\to \mathbb{R}_+\times [C(\overline{\Omega})]^2$ by 
\begin{equation}\label{FRD1}
    \mathcal{F}_{d_I,d_S}(l)=\Big(\mathcal{N}_{d_{I},d_S}(l),l(1-d_Iu^l),d_Slu^l\Big) \quad \forall\ l\ge l^*,
\end{equation}
where $u^l$ is the unique nonnegative stable solution of \eqref{Eq1}. Recalling that $\mathcal{R}_0=N/(|\Omega|l^*)$, then $(\frac{N}{|\Omega|},0)=(l^*\mathcal{R}_0,0)$ is the unique DFE solution of \eqref{e1}. Hence $(l^*,0)$ is the unique DFE solution of \eqref{e1} when $\mathcal{R}_0=1$. Observe also that $\mathcal{F}_{d_I,d_S}(l^*)=(1,l^*,0)$. By Lemma \ref{lem3}, system \eqref{e1} has an EE solution $(S,I)$ for some $\mathcal{R}_0>0$ if and only if $\mathcal{R}_0=\mathcal{N}_{d_I,d_S}(l)$ and $(S,I)=(l(1-d_Iu^l),ld_Su^l)$ for some $l>l^*$.  Therefore, as $\mathcal{R}_0$ increases from zero to infinity, the EE solutions of \eqref{e1} are parametrized by the curve $\mathcal{F}_{d_I,d_S}$. This curve is simple and unbounded since the mapping $l\mapsto lu^l$ is strictly increasing with $\|lu^l\|_{\infty}\to\infty$ as $l\to\infty$.  Furthermore, since $\frac{d\mathcal{N}_{d_I,d_S}(l^*)}{dl}<0$, then the curve parametrized by $\mathcal{F}_{d_I,d_S}$ bifurcates from the left at $\mathcal{R}_0=1$.


\end{proof}

Note that the expression at the right hand side of \eqref{Eq1} is analytic in the variables $l$ and $u$. Hence, we can employ implicit function theorem \cite[Page 15]{Dan_Henry} and linear stability of  $u^l$ to derive that $u^l$ is analytic in $l>l^*$. Hence, thanks to Lemma \ref{lem2} and the limit \eqref{GH2},  we have the following: 

\begin{lem}\label{lem4} Fix $d_I>0$ and $d_S>0$. Consider the mapping $\mathcal{N}_{d_I,d_S}$  defined by \eqref{N-equation} on $[l^*,\infty)$. Then $\mathcal{N}_{d_I,d_S}$ is continuously differentiable on $[l^*,\infty)$ and analytic on $(l^*,\infty)$. Furthermore, if $\frac{d\mathcal{N}_{d_I,d_S}(l^*)}{dl}\ne 0$, then there exist $m$ numbers $l^*_1=l^*<l^*_2<\cdots<l^*_m<l^*_{m+1}=\infty$, $m\ge 1$, such that $\mathcal{N}_{d_I,d_S}$ is strictly monotone on $[l^*_i,l^*_{i+1})$ for each $i=1,\cdots,m$;  $\mathcal{N}_{d_I,d_S}$ is strictly increasing on $[l^*_{m},\infty)$; and if $m\ge 2$, $\mathcal{N}_{d_I,d_S}$ changes its monotonicity  at each $l_i^*$, $i=2,\cdots,m$.     
\end{lem}

Next, we give  a proof of Theorem \ref{T2-2}.

\begin{proof}[Proof of Theorem \ref{T2}] Suppose that \eqref{T2-eq1} holds.   Then \begin{equation}\label{ZXZ2}
\lim_{l\to\infty}\mathcal{N}_{d_I}(l)= \overline{(\gamma/\beta)}\mathcal{R}_1<1=\mathcal{N}_{d_I}(l^*).
\end{equation}
As a result,  there exist some $\tilde{l}_0^*\gg l^*$ such that   $\tilde{M}^*_{d_I}:= \sup_{l\ge \tilde{l}_0^*}\mathcal{N}_{d_I}(l)<\mathcal{N}_{d_I}(l^*)$. We first set $\tilde{d}_3^*:=|\Omega|l^*({\mathcal{N}_{d_I}(l^*)-\tilde{M}^*_{d_I}})/({\tilde{l}_0^*\int_{\Omega}u^{\tilde{l}_0^*}})$. Next, note from  \eqref{T2-2-eq1} and  \eqref{v-l-eq2} that  
\begin{equation}\label{ZXZ1}
\frac{d\mathcal{N}_{d_I}(l^*)}{dl}=  \Big(1-{\big(\overline{\varphi_1}\big)\big(\overline{\beta\varphi_1^2}\big)}\big/{\overline{\beta\varphi_1^3}}\Big)\mathcal{R}_1>0.
\end{equation}
Hence, there is $0<d_3^*\le \tilde{d}_3^*$ such that $\frac{d\mathcal{N}_{d_I,d_S}(l^*)}{dl}>0$ for all $0\le d_S\le d_3^*$. Now fix $0<d_S<d_3^*$. Let $m\ge 1$ be given by Lemma \ref{lem4}. Hence $\mathcal{N}_{d_I,d_I}$ is strictly increasing on $[l^*,l_2^*)$ and on $[l_m^*,\infty)$.  Observing that $\mathcal{N}_{d_I,d_S}(l^*)=\mathcal{N}_{d_I}(l^*)=1$ and 
 $$
 \mathcal{N}_{d_I,d_S}(\tilde{l}_0^*)=\mathcal{N}_{d_I}(\tilde{l}_0^*)+\frac{\tilde{l}_0^*d_S}{l^*|\Omega|}\int_{\Omega}u^{\tilde{l}_0^*}<\mathcal{N}_{d_I}(\tilde{l}_0^*)+\frac{\tilde{l}_0^*\tilde{d}_3^*}{|\Omega|l^*}\int_{\Omega}u^{\tilde{l}_0^*}=\mathcal{N}_{d_I}(\tilde{l}_0^*)+\mathcal{N}_{d_I,d_S}(l^*)-\tilde{M}^*_{d_I}\le \mathcal{N}_{d_I,d_S}(l^*),
 $$
 then we must have that $m\ge 3.$ Note that the simple connected curve $\mathcal{C}$  parametrized by \eqref{FRD1} as in the proof of Theorem \ref{T2-2} consists of the EE solutions of \eqref{e1}. This time around, since $\mathcal{N}_{d_I,d_S}$ is increasing on $[l^*,l_1^*]$, then $\mathcal{C}$ bifurcates from the right at  $\mathcal{R}_0=1$.

 \quad Next, since $\mathcal{N}_{d_I,d_S}$ is strictly increasing on $[l_1^*,l_2^*]$ and $[l^*_m,\infty)$, and $\mathcal{N}_{d_I,d_S}(\tilde{l}_0)<\mathcal{N}_{d_I,d_S}(l_1^*)$, $\mathcal{R}_{0,1}^{d_S}:=\min\{\mathcal{N}_{d_I,d_S}(l^*_i) : i=2,\cdots,3 \}$ is the global minimum value of $\mathcal{N}_{d_I,d_S}$ and is achieved at some $l_{i_0}^*$, $i_0=3,\cdots,m$.  So, by Lemma \ref{lem3}, system \eqref{e1} has no EE solution for $\mathcal{R}_0<\mathcal{R}_{0,1}^{d_S}$ and $ 
 (l_{i_0}(1-d_Iu^{l_{i_0}}),d_Sl_{i_0}u^{l_{i_0}})
 $  is an EE solution of \eqref{e1} for $\mathcal{R}_0=\mathcal{R}_{0,1}^{d_S}$.  Thus, {\rm (i)} and {\rm (ii)} are proved.  Next, set $\mathcal{R}_{0,2}^{d_S}=\mathcal{N}_{d_I,d_S}(l_2^*)$, and $\mathcal{R}_{0,3}^{d_S}=\max\{\mathcal{N}_{d_I,d_S}(l^*_i):i=1,\cdots,m\}$

\quad {\rm (iii)} First, suppose that $\mathcal{R}_0\in (\mathcal{R}_{0,1}^{d_S},1]$. By the intermediate value theorem,  it follows as in \eqref{l-high-def} and \eqref{l-low-def} that both $ l_{\rm high}(d_S)>l_{m}^*$ and $l_{\rm low}(d_S)\in (l^*_{2},l_m^*) $ are well defined. Moreover, $(S_{\rm high},I_{\rm high})$ and $(S_{\rm low}, I_{\rm low})$ defined as in \eqref{high-ee} and \eqref{low-ee}, respectively, are two distinct EE solutions of \eqref{e1}.

\quad Next, suppose that $\mathcal{R}_0=\mathcal{R}_{0,2}^{d_S}$.  By the intermediate value theorem, since $$
 \mathcal{N}_{d_I,d_S}(l_m^*)<\mathcal{N}_{d_I,d_S}(l_2^*)=\mathcal{R}_{0,2}^{d_S}<\mathcal{R}_0,
 $$
 and $\mathcal{N}_{d_I,d_S}(l)\to \infty$ as $l\to \infty$, there is  $ l_{\rm high}(d_S)>l_{m}^*$ such that $\mathcal{N}_{d_I,d_S}(l_{\rm high})=\mathcal{R}_0$. Hence,   $({l}_{\rm high}(1-d_Iu^{{l}_{\rm high}}),d_S{l}_{\rm high}u^{{l}_{\rm high}})$  and $({l}^*_{2}(1-d_Iu^{{l}^*_{2}}),d_S{l}^*_{2}u^{{l}^*_{2}})$  are two distinct EE solutions of \eqref{e1}. This completes the proof of {\rm (iii)}

\quad   {\rm (iv)} Suppose that $1<\mathcal{R}_0<\mathcal{R}_{0,2}^{d_S}$. Observe that 
$$
\mathcal{N}_{d_{I},d_S}(l_1^*)=1<\mathcal{R}_0<\mathcal{R}_{0,2}^{d_S}=\mathcal{N}_{d_I,d_S}(l_2^*) \quad \text{and}\quad \mathcal{N}_{d_I,d_S}(l_{i_0})<\mathcal{N}_{d_I,d_S}(l_1^*)<\mathcal{R}_{0}<\mathcal{N}_{d_I,d_S}(l_2^*).
$$
 Hence, since $\mathcal{N}_{d_I,d_S}(l)\to \infty$ as $l\to \infty$, we can employ the intermediate theorem to deduce the existence of minimal numbers  $l_{\rm low,1}\in (l_1^*,l_2^*)$, $l_{\rm low,2}\in ({l}^*_{2},l_{i_0}^*)$, and a maximal number $l_{\rm high}>l_{i_0}^*$ such that $(S^i_{\rm low},I^i_{\rm low}):=(l_{\rm low,i}(1-d_Iu^{l_{\rm low,i}}),d_Sl_{\rm low,i}u^{l_{\rm low,i}})$, $i=1,2$, and $(S{\rm high},I_{\rm high}):=(l_{\rm high}(1-d_Iu^{l_{\rm high}}),d_Sl_{\rm high}u^{l_{\rm high}})$ are three different EE solutions of \eqref{e1}.  Clearly, $(S^1_{\rm low},I^1_{\rm low})$ and $(S_{\rm high},I_{\rm high})$ are the minimal and maximal EE solutions of \eqref{e1} in the sense of \eqref{T0-eq2} and \eqref{T0-eq1}, respectively.

 \quad {\rm (v)}  If $\mathcal{R}_0>\mathcal{R}_{0,2}^{d_S}$, then since $\mathcal{N}_{d_I,d_S}(l)\to \infty$ as $l\to \infty$,  it follows from the intermediate value theorem that there is $l_{\rm high}>l^*_2$, that $(l_{\rm high}(1-d_Iu^{l_{\rm high}}),d_Sl_{\rm high}u^{l_{\rm high}})$  is an EE solution of \eqref{e1}.  Now, suppose that  $\mathcal{R}_0>\mathcal{R}_{0,3}^{d_S}$. Then, since  $\mathcal{N}_{d_I,d_S}$ is strictly increasing on $[l^*_m,\infty)$, $\mathcal{N}_{d_I,d_S}(l^*_m)<\mathcal{R}_0$, and $\mathcal{N}_{d_I,d_S}(l)\to \infty$ as $l\to\infty$, there is a unique $\hat{l}>{l}^*_{m}$ such that $\mathcal{N}_{d_I,d_S}(\hat{l})=\mathcal{R}_0$ and $\mathcal{N}_{d_I,d_S}(l)<\mathcal{R}_0$ for all $l\in[l^*_m,\hat{l})$. Moreover, observe that $$
 \mathcal{N}_{d_I,d_S}(l)\le \mathcal{R}_{0,3}^{d_S}<\mathcal{R}_0,\quad \forall\ l\in[l^*_1,l_m^*].
 $$
 Therefore, $(\hat{l}(1-d_Iu^{\hat{l}}),d_S\hat{l}u^{\hat{l}})$ is the unique EE solution of \eqref{e1}. 

 \quad Since $\mathcal{N}_{d_I,d_S}$ is strictly monotone increasing in $d_S$, we see that $\mathcal{R}_{0,i}^{d_S}$ is strictly increasing in $d_S$ for each $i=1,\cdots,3$. Clearly, from the definition of $\mathcal{R}_{0,1}^{d_S}=\min_{l\ge l^*}\mathcal{N}_{d_I,d_S}(l)$, we get that $\mathcal{R}_{0,1}^{d_S}\to \inf_{l\ge l^*}\mathcal{N}_{d_I}(l)=\mathcal{R}_0^{\rm low}$. Finally, from \eqref{ZXZ1}, we can find $\tilde{l}^{**}_{0}>l^*$ such that $\mathcal{N}_{d_I}$ is strictly increasing on $[l^*,\tilde{l}^{**}_0]$. As a result, we get that $\mathcal{N}_{d_I,d_S}(l)=\mathcal{N}_{d_I}(l)+(d_Sl\int_{\Omega}u^l)/(|\Omega|l^*)$ is strictly increasing on $[l^*,\tilde{l}^{**}_0]$ for every $d_S>0$ since $u^l$ is strictly increasing in $l>l^*$. Therefore, for every $0<d_S<d_3^*$, $l^*_2\ge \tilde{l}^{**}_0$ and $1=\mathcal{N}_{d_I,d_S}(l^*)<\mathcal{N}_{d_I}(\tilde{l}^{**}_0)<\mathcal{N}_{d_I,d_S}(l^*_2)=\mathcal{R}_{0,2}^{d_S}$. As a result, $\mathcal{R}_{0,2}^*=\lim_{d_S\to 0}\mathcal{R}_{0,2}^{d_S}\ge \mathcal{N}_{d_I}(\tilde{l}^{**}_0)>1$.

\end{proof}

We complete this section with  a proof of Theorem \ref{T4}.

\begin{proof}[ Proof of Theorem \ref{T4}] Fix $d_I>0$ and suppose that $\mathcal{R}_0^{\rm low}<1$.

\quad {\rm (i)} Fix $\mathcal{R}_0^{\rm low}<\mathcal{R}_0<1$  and let $d_1^*$ be given by Theorem \ref{T0}. For every $0<d_S<d_1^*$, let $l_{\rm high}(d_S)$ and $l_{\rm low}(d_S)$ be defined by \eqref{l-high-def} and \eqref{l-low-def}, respectively. First, we  claim that 
    
    \begin{equation}\label{a-rev-2}
        l_{\rm low}(d_{S,1})<l_{\rm low}(d_{S,2})\quad \forall\ 0<{d_{S,1}}<d_{S,2}<d_1^*.
    \end{equation}
    Indeed, fix $0<d_{S,1}<d_{S,2}<d_1^*$. Then, since $\mathcal{N}_{d_I,d_{S,2}}(l_{\rm low}(d_{S,2}))=\mathcal{R}_0$,
    \begin{align*}
        \mathcal{N}_{d_I,d_{S,1}}(l_{\rm low}(d_{S,2}))=\mathcal{N}_{d_I,d_{S,2}}(l_{\rm low}(d_{S,2}))
        -(d_{S,2}-d_{S,1})\frac{l_{\rm low}(d_{S,2})}{l^*|\Omega|}\int_{\Omega}u^{l_{\rm low}(d_{S,2})}<\mathcal{R}_0.
        \end{align*}
    Hence, by \eqref{l-low-def} and \eqref{a-rev-1}, we have that \eqref{a-rev-2} holds. Next, we claim that

    \begin{equation}\label{a-rev-3}
        l_{\rm high}(d_{S,1})>l_{\rm high}(d_{S,2})\quad \forall\ 0<{d_{S,1}}<d_{S,2}<d_1^*.
    \end{equation}
    Indeed, fix $0<d_{S,1}<d_{S,2}<d_1^*$. Then, since $\mathcal{N}_{d_I,d_{S,2}}(l_{\rm high}(d_{S,2}))=\mathcal{R}_0$ 
    \begin{align*}
        \mathcal{N}_{d_I,d_{S,1}}(l_{\rm high}(d_{S,2}))=\mathcal{N}_{d_I,d_{S,2}}(l_{\rm high}(d_{S,2}))
        -(d_{S,2}-d_{S_1})\frac{l_{\rm high}(d_{S,2})}{|\Omega|l^*}\int_{\Omega}u^{l_{\rm high}(d_{S,2})}<\mathcal{R}_0.
    \end{align*}
    Hence, by \eqref{l-high-def} and \eqref{a-rev-0}, we have that \eqref{a-rev-3} holds.
    
    \quad Thanks to \eqref{a-rev-2} and \eqref{a-rev-3},  the following limits exist
    \begin{equation*}
        l_{\rm low}^*=\lim_{d_S\to 0^+}l_{\rm low}(d_S)=\inf_{0<d_S<d_1^*}l_{\rm low}(d_S)<l(\mathcal{R}_0,N)
    \end{equation*}
    and 
    \begin{equation*}
        l_{\rm high}^*:=\lim_{d_S\to 0^+}l_{\rm high}(d_S)=\sup_{0<d_S<d_1^*}l_{\rm high}(d_S)>l(\mathcal{R}_0,d_I).
    \end{equation*}

    \quad {\rm (i-1)} Suppose that $ \mathcal{R}_0/\mathcal{R}_1<\overline{{\gamma}/{\beta}}$ and we establish that \eqref{T0-eq4} holds.    In the current case, we first proceed by contradiction to show that 
    \begin{equation}\label{G1}
        l_{\rm high}^*<\infty.
    \end{equation}
    Indeed, if \eqref{G1} were false, then  $l_{\rm high}(d_{S})\to \infty$ as $d_S\to 0$. As, a result, it follows from \eqref{lem2-eq2} that $\int_{\Omega}l_{\rm high}(d_{S})(1-d_Iu^{l_{\rm high}(d_{S})})\to|\Omega|\overline{({\gamma}/{\beta})}$ as $d_S\to 0$. 
    Moreover, since
    \begin{equation*}
        \mathcal{R}_0=\mathcal{N}_{d_I,d_{S}}(l_{\rm high}(d_{S}))>\frac{1}{|\Omega|l^*}\int_{\Omega}l_{\rm high}(d_{S})(1-d_Iu^{l_{\rm high}(d_{S})})\quad \forall\ 0<d_S<d_1^*,
    \end{equation*}
    we obtain that 
    $  
    \mathcal{R}_0\ge \frac{1}{|\Omega|l^*}\lim_{d_S\to 0}\int_{\Omega}l_{\rm high}(d_{S})(1-d_Iu^{l_{\rm high}(d_{S})})=\overline{(\gamma/\beta)}\mathcal{R}_1,
    $ 
    which gives a contradiction. Thus, \eqref{G1} holds.  This shows that there is a positive constant $C_1=C_1(\mathcal{R}_0,d_I)$ such that 
    \begin{equation}\label{G3}
        l_{\rm high}(d_S)\le C_1 \quad \forall\ 0<d_S<\frac{d_1^*}{2}.
    \end{equation}
    In particular,
\begin{equation}\label{G7}
    I_{\rm low}<I_{\rm high}=d_Sl_{\rm high}(d_S)u^{l_{\rm high}(d_S)}\le \frac{C_1}{d_I}d_S\quad \forall\ 0<d_S<\frac{d_1^*}{2}.
\end{equation}
Next, we claim that 
\begin{equation}\label{G4}
    l_{\rm low}^*>l^*.
\end{equation}
If \eqref{G4} were false, then  $l_{\rm low}(d_{S})\to l^*$ as $d_S\to 0$. This in turn implies that 
\begin{align*}
    \mathcal{R}_0=\mathcal{N}_{d_{I}}(l_{\rm low}(d_{S}))+\frac{d_{S}l_{\rm low}(d_{S})}{|\Omega|l^*}\int_{\Omega}u^{l_{\rm low}(d_{S})} 
     \to  1\quad \text{as}\ d_{S}\to 0,
\end{align*}
which contradicts our initial assumption that $\mathcal{R}_0\ne 1$. Therefore, \eqref{G4} holds. Thus, there is $C_2=C_2(\mathcal{R}_0,d_I)>l^*$ such that 
\begin{equation}\label{G5}
    C_2\le l_{\rm low}(d_S)\le l(\mathcal{R}_0,d_I)\quad \forall\  0<d_S<\frac{d_1^*}{2}.
\end{equation}
As a result, for  all $0<d_S<\frac{d_1^*}{2}$, we obtain that 
\begin{align}\label{G6}
I_{\rm low}=&d_{S}l_{\rm low}(d_S)u^{l_{\rm low}(d_S)}
\ge C_2d_Su^{C_2}
\ge  C_2u^{C_2}_{\min}d_S\quad.
\end{align}
Combining \eqref{G6} and \eqref{G7} we derive that \eqref{T0-eq5} holds. 

\quad Next, since $l_{\rm high}^*>l_{\rm low}^*>l^*$, then by Lemma \ref{lem2}-{\rm (ii)}, we have that 
$S_{\rm low}\to l_{\rm low}^*(1-d_Iu^{l_{\rm low}^*})$ and $S_{\rm high}\to l_{\rm high}^*(1-d_Iu^{l_{\rm high}^*})
$ as $d_S\to 0$ in $C^{1}(\overline{\Omega})$. On the other hand, since $N=\int_{\Omega}(S+I)$, we have that 
$$ 
N=l_{\rm low}^*\int_{\Omega}(1-d_Iu^{l_{\rm low}^*}) \quad \text{and}\quad N=l_{\rm high}^*\int_{\Omega}(1-d_Iu^{l_{\rm high}^*}).
$$
Finally, since  $l_{\rm low}^*<l(\mathcal{R}_0,d_I)<l_{\rm high}^*$, then $u^*_{\rm low}:=u^{l_{\rm low}^*}<u^{l_{\rm high}^*}=:u^*_{\rm high}$.

\quad {\rm (i-2)} Suppose that $ \mathcal{R}_0/\mathcal{R}_1>\overline{{\gamma}/{\beta}}$. Note that the proof of \eqref{G5} only relies on the fact that $\mathcal{R}_0\ne 1$. Hence, $(S_{\rm low},I_{\rm low})$ satisfies \eqref{T0-eq4} and \eqref{T0-eq5} as $d_S\to 0$. We claim that 
\begin{equation}\label{G8}
    l_{\rm high}^*=\infty.
\end{equation}
Indeed, since $\mathcal{R}_0>\overline{(\gamma/\beta)}\mathcal{R}_1=\lim_{l\to\infty}\mathcal{N}_{d_I}(l)$, then for every $m>1$, there is $l_m(\mathcal{R}_0,d_I)>m$ such that 
$$ 
\mathcal{N}_{d_I}(l_m(\mathcal{R}_0,d_I))<\mathcal{R}_0.
$$
Therefore, taking this time $d_m(\mathcal{R}_0,d_I):=\frac{(\mathcal{R}_0-\mathcal{N}_{d_I}(l_m(\mathcal{R}_0,d_I)))|\Omega|l^*}{l_m(\mathcal{R}_0,d_I)\int_{\Omega}u^{l_m(\mathcal{R}_0,d_I)}}$, for every $0<d_S<d_m(\mathcal{R}_0,d_I)$, we can employ similar arguments as in \eqref{C2} and the intermediate value theorem to conclude that there is ${l}_m(d_S)>l_m(\mathcal{R}_0,d_I)$ such that $\mathcal{N}_{d_I,d_S}(l_{m}(d_S))=\mathcal{R}_0$. This shows that 
$$ 
l_{\rm high}(d_S)\ge l_{m}(\mathcal{R}_0,d_S)>m\quad \forall\ 0<d_S<d_{m}(\mathcal{R}_0,d_I).
$$
Letting $m\to \infty$ in this inequality leads to \eqref{G8}. Thus, since $l(1-d_Iu^l)\to \frac{\gamma}{\beta}$ as $l\to\infty$ in $C(\overline{\Omega})$ (see \eqref{lem2-eq1}), we conclude that 
$ S_{\rm high}=l_{\rm high}(d_S)(1-d_Iu^{l_{\rm high}(d_S)})\to \frac{\gamma}{\beta}$ as $ d_S\to 0$
uniformly in $C(\overline{\Omega})$. Observing that
$  
d_Sl_{\rm high}(d_S)={\big(N-\int_{\Omega}S_{\rm high}\big)}/{\int_{\Omega}u^{l_{\rm high}(d_S)}}\to {d_I(N-\int_{\Omega}\frac{\gamma}{\beta})}/{|\Omega|}$ as $ d_S\to 0$, where we have used the fact that $u^{l_{\rm high}(d_S)}\to 1/d_I$ as $d_S\to 0$ uniformly on $\Omega$, then 
$ 
I_{\rm high}\to\left(N-\int_{\Omega}\frac{\gamma}{\beta}\right)/|\Omega|$ {as}\ $d_S\to 0
$ 
in $C(\overline{\Omega})$.

\quad {\rm (ii)} In addition, suppose that \eqref{T2-eq1} holds. Let $d_3^*$ and $\mathcal{R}_{0,2}^*$ be given by Theorem \ref{T2} and  fix $1<\mathcal{R}_0<\mathcal{R}^*_{0,2}$. Hence, $1<\mathcal{R}_0<\mathcal{R}_{0,2}^{d_S}$ for every $0<d_S<d_3^*$. Note from the proof of Theorem \ref{T2}-{\rm (iv)}, that for every $0<d_S<d_3^*$,  $(S^1_{\rm low},I^1_{\rm low})$ and $(S_{\rm high},I_{\rm high})$ are the minimal and maximal EE solutions of \eqref{e1} in the sense of \eqref{T0-eq2} and \eqref{T0-eq1}, respectively.  Observing here also from \eqref{T2-eq1} that $\mathcal{R}_0>\overline{({\gamma}/{\beta})}\mathcal{R}_1$, then by the similar argument as in {\rm (i-2)}, we have that  $(S_{\rm high},I_{\rm high})$ has the asymptotic profiles \eqref{T0-eq7} as $d_S\to 0$.  Next, observe that the constant number $\tilde{l}^*_0$ obtained in the proof of Theorem \ref{T2} depends only on $d_I$ and satisfies $l^*_2<\tilde{l}^*_{0}$ for every $0<d_S<d_3^*$. Therefore, $l_{\rm low,1}<\tilde{l}^*_{0}$ for every $0<d_S<d_3^*$. Finally, since $\mathcal{R}_0\ne 1$, we can proceed by the similar arguments leading to \eqref{G4} to obtain that $C_2:=\liminf_{d_S\to0}l_{\rm low,1}>0$. In view of the preceding details, we see that $l_{\rm low,1}$ satisfies inequalities \eqref{G5} with $l(\mathcal{R}_0,d_I)$ being replaced by $\tilde{l}^*_0$. So, $I_{\rm low,1}$ satisfies \eqref{T0-eq4}. Furthermore, up to a subsequence,    $S_{\rm low}^1$ satisfies \eqref{T0-eq5} as $d_S\to 0$.
\end{proof}


\section{Appendix} 
Let $\lambda_0=0<\lambda_1\le \lambda_2\le \cdots\le \lambda_m\le\cdots$ satisfying $\lambda_m\to \infty$ as $m\to\infty$ denote the eigenvalues of 
\begin{equation*}
    \begin{cases}
        0=\Delta \phi +\lambda\phi & x\in\Omega,\cr
        0=\partial_{\vec{n}}\phi & x\in\partial\Omega.
    \end{cases}
\end{equation*}
Let $\{\phi_m\}_{m\ge 0}$ be the orthonormal basis  of $L^2(\Omega)$ where $\phi_m$ is an eigenfunction associated with $\lambda_m$ for each $m\ge 0$. Next, consider the Banach space $\tilde{\mathcal{Z}}:=\{w\in L^2(\Omega) : \overline{w}=0\}={\rm span}(\phi_0)$.  For every $q\ge 2$, the restriction of the Laplace operator on ${\rm Dom}_q\cap\tilde{\mathcal{Z}}$ to $\tilde{\mathcal{Z}}_q:=L^q(\Omega)\cap \tilde{\mathcal{Z}}$ is invertible. Lettting $C_q^*:=\|\Delta^{-1}_{|{\rm Dom}_q\cap\tilde{\mathcal{Z}}}\|$,  for every $w\in \tilde{\mathcal{Z}}_q$, the unique solution $W\in {\rm Dom}_q\cap\tilde{\mathcal{Z}}$  of 
\begin{equation}\label{appen-1}
    \begin{cases}
        0=\Delta W+w & x\in\Omega,\cr 
        0=\partial_{\vec{n}}W & x\in\partial\Omega,\cr
        0=\overline{W}
    \end{cases}
\end{equation}
satisfies
\begin{equation}\label{appen-2}\|W\|_{W^{2,q}(\Omega)}\le C^*_q\|w\|_{L^q(\Omega)}.\end{equation}
Fix  a H\"older continuous and non-constant function $h$ on $\overline{\Omega}$,  and positive constants $k>0$ and $d_I>0$. Define
\begin{equation}\label{UO1}
    l_0^*=k,\quad \tilde{\varphi}_0={1}/{|\Omega|},\quad \quad \text{and}\quad l_1^*=\overline{h}.
\end{equation}
Let $\tilde{\varphi}_1$ be  the unique solution \eqref{appen-1} with $w_1:=l_0^*(l_1^*-h)\tilde{\varphi}_0/d_I$. Note that $\tilde{\varphi}_1$ is well defined since $\int_{\Omega}w_1=0.$ Next, define 
\begin{equation}\label{l-star-2-eq}
l_2^*=\Big(\overline{(h-l_1^*)h}+k|\Omega|\overline{(h-l_1^*)\tilde{\varphi}_1}\Big)/k,
\end{equation}
and $\tilde{\varphi}_2$ is the unique solution of \eqref{appen-1} with $w_2:=(kl_2^*\tilde{\varphi}_0 -h(h-l_1^*)\tilde{\varphi}_0-k(h-l_1^*)\tilde{\varphi}_1)/d_I.$ Note also that $\tilde{\varphi}_2$ is well defined since $\int_{\Omega}w_2=0$.  Throughout the rest of this section, whenever $h$, $k$ and $d_I$ are given, we shall suppose that $l^*_0$, $l^*_1$, $l_2^*$, $\tilde{\varphi}_0$, $\tilde{\varphi}_1$, and $\tilde{\varphi}_2$ are defined as above.

\begin{prop}\label{appen-prop2} Fix 
$k>0$ and $d_I>0$ and suppose that $h=c_m\phi_m$ for some $m\ge 1$ where $c_m$ is a nonzero constant. Then
    $\tilde{\varphi}_1=-(kh)/{(|\Omega|d_I\lambda_m)}$ and $ l_2^*=\Big(\Big(1-{k^2}/({d_I\lambda_m})\Big)\overline{h^2}\Big)/k$.


    
\end{prop}

\begin{proof}It can be verified by inspection.


\end{proof}



\begin{prop}\label{appen-prop1} Fix $d_I>0$,
$k>0$, and $h$  H\"older continuous and non-constant function $h$ on $\overline{\Omega}$. For every, $0<\varepsilon<\varepsilon_{k,h}:=\frac{k}{\|h\|_{\infty}}$, define \begin{equation}\label{appen-3}
    \beta_{k,h,\varepsilon}=k+\varepsilon h,
\end{equation} 
and  let $l^*(\varepsilon)$ denote the principal eigenvalue of the weighted linear elliptic equation
\begin{equation}\label{prop-eq2}
    \begin{cases}
        0=d_I\Delta \tilde{\varphi}+\beta_{k,h,\varepsilon}(l(\varepsilon)-\beta_{k,h,\varepsilon})\tilde{\varphi} & x\in\Omega,\cr
        0=\partial_{\vec{n}}\tilde{\varphi} & x\in\Omega.
    \end{cases}
\end{equation}
Denote by $\tilde{\varphi}(\cdot;\varepsilon)$ the unique positive solution of \eqref{prop-eq2} satisfying $\int_{\Omega}\tilde{\varphi}(\cdot;\varepsilon)=1$. Next, define 
\begin{equation}
    \tilde{\varphi}_3(\cdot;\varepsilon)=({\tilde{\varphi}(\cdot;\varepsilon)-\tilde{\varphi}_0-\varepsilon\tilde{\varphi}_1-\varepsilon^2\tilde{\varphi}_2})/{\varepsilon^3} \quad \text{and}\quad l^*_3(\varepsilon)=({l^*(\varepsilon)-{l}^*_0-\varepsilon l^*_1-\varepsilon^2l^*_2})/{\varepsilon^3}.
\end{equation}
 It holds that 
\begin{equation}\label{prop-eq4}
    \limsup_{\varepsilon\to 0}\|\tilde{\varphi}_3(\cdot;\varepsilon)\|_{C^1(\overline{\Omega})}<\infty \quad \text{and}\quad \limsup_{\varepsilon\to 0}|l_3^*(\varepsilon)|<\infty.
\end{equation}
    
\end{prop}
\begin{proof} By computations, we have that  $\tilde{\varphi}_3$ and $l_3^*$ satisfy
\begin{equation}\label{proof-prop-app-eq1}
    \begin{cases}
        0=d_I\Delta\tilde\varphi_3+  \big(kl_2^*+h(l_1^*-h)\big)(\tilde{\varphi}_1+\varepsilon\tilde{\varphi}_2+\varepsilon^2\tilde{\varphi}_3)+(\beta_{k,h,\varepsilon}l_3^*+hl_2^*)\tilde{\varphi}+ k(l_1^*-h)(\tilde\varphi_2+\varepsilon\tilde{\varphi}_3)& x\in\Omega,\cr   0=\partial_{n}\tilde\varphi_3 & x\in\partial\Omega,\cr
        0=\int_{\Omega}\tilde\varphi_3
    \end{cases}
\end{equation}
Now, fix $q\gg 1$ such that $W^{2,q}(\Omega)$ is continuously embbeded in $C^1(\overline{\Omega})$. Then, by the similar arguments leading to \eqref{YTYT1} and the fact that, there exist $0<\varepsilon_{q,k,h}\ll 1$  and $M_{q,k,h}>0$ such that 
\begin{equation}\label{proof-prop-app-eq2}
    \|\tilde{\varphi}_3(\cdot;\varepsilon)\|_{W^{2,q}(\Omega)}\le M_{q,k,h}(1+|l_3^*|)\quad 0<\varepsilon\le  \varepsilon_{q,k,h}.
\end{equation}
Setting $A_1=kl_2^*+h(l_1^*-h)$ and $A_2=k(l_1^*-h)$, integrating \eqref{proof-prop-app-eq1} and rearranging the terms yield:
\begin{equation*}
    l_3^*\int_{\Omega}\beta_{k,h,\varepsilon}\tilde{\varphi}=-\int_{\Omega}A_1(\tilde{\varphi}_1+\varepsilon\tilde{\varphi}_2)-l_2^*\int_{\Omega}h\tilde{\varphi}-\int_{\Omega}A_2\tilde{\varphi}_2-\varepsilon \int_{\Omega}(A_2+\varepsilon A_1)\tilde{\varphi}_3.
\end{equation*}
Hence, setting $B:=|\Omega|\|A_1\|_{\infty}(\|\tilde{\varphi}_1\|_{\infty}+\|\tilde{\varphi}_2\|_{\infty})+\sum_{i=1}^2\|A_i\|_{\infty}+\sum_{i=1}^3\|\tilde{\varphi}_i\|_{\infty} $ and using \eqref{proof-prop-app-eq2}, for every $0<\varepsilon<\varepsilon_{q,k,h}$, we have
\begin{align*}
    |l_3^*(\varepsilon)|\int_{\Omega}\beta_{k,h,\varepsilon}\tilde{\varphi}\le & |\Omega|\|A_1\|_{\infty}(\|\tilde{\varphi}_1\|_{\infty}+\varepsilon\|\tilde{\varphi}_2\|_{\infty})+|\Omega||l_2^*|\|\tilde{\varphi}\|_{\infty}+\varepsilon(\|A_1\|_{\infty}+\varepsilon\|A_2\|_{\infty})\|\tilde\varphi_{3}\|_{\infty}\cr 
    \le & B+|\Omega||l_2^*|(\|\tilde{\varphi}_0\|_{\infty}+\varepsilon\|\tilde{\varphi}_1\|_{\infty}+\varepsilon^2\|\tilde{\varphi}_2\|_{\infty}+\varepsilon^3\|\tilde{\varphi}_3\|_{\infty}) +\varepsilon B\|\tilde{\varphi}_3\|_{\infty}\cr 
    \le & (1+|\Omega||l_2^*|)B+\varepsilon M_{q,k,k}(B+|\Omega||l_2^*|)(1+|l_3^*(\varepsilon)|).
\end{align*}
Equivalently,
\begin{equation}\label{proof-prop-app-eq3}
    |l_3^*|\Big(\int_{\Omega}\beta_{k,h,\varepsilon}\tilde{\varphi}-\varepsilon M_{q,k,h}(B+|\Omega||l_2^*|)\Big)\le (1+|\Omega||l_2^*|)B+\varepsilon M_{q,k,h}(B+|\Omega||l_2^*|)\quad 0<\varepsilon<\varepsilon_{q,h,k}.
\end{equation}
Next, observing that 
$$
\int_{\Omega}\beta_{k,h,\varepsilon}\tilde{\varphi}\ge (k-\varepsilon\|h\|_{\infty})\int_{\Omega}\tilde{\varphi}=k-\varepsilon\|h\|_{\infty}\quad 0<\varepsilon<\frac{k}{\|h\|_{\infty}},
$$
it follows from \eqref{proof-prop-app-eq3} that 
$ \limsup_{\varepsilon\to 0}|l_3^*(\varepsilon)|\le \frac{(1+|\Omega||l_2^*|)B}{k}$,  which in view of \eqref{proof-prop-app-eq2} yields also that 
$ \limsup_{\varepsilon\to 0}\|\tilde{\varphi}_3(\cdot;\varepsilon)\|_{C^1(\overline{\Omega})}<\infty$ 
since $W^{2,q}(\Omega)$ is continuously embedded in $C^1(\overline{\Omega})$.

\end{proof}

\begin{prop}\label{appen-prop3} Fix 
$k>0$ and $d_I>0$ and suppose that $h=c_m\phi_m$ for some $m\ge 1$. For every $0<\varepsilon\ll 1$, let $\gamma_{k,h,\varepsilon}=\beta_{k,h,\varepsilon}^2$ where $\beta_{h,k,\varepsilon}$ is defined by \eqref{appen-3}. Then $\mathcal{R}_1=\frac{1}{l^*(\varepsilon)}$. Moreover, for sufficiently small values of $\varepsilon$, it holds that 
\begin{equation}\label{TTH1}
     {1}/{\mathcal{R}_1}-\overline{{\gamma_{k,h,\varepsilon}}/{\beta_{k,h,\varepsilon}}}\begin{cases}
         >0 & \text{if}\quad d_I\lambda_m>k^2,\cr
         <0 & \text{if}\quad d_I\lambda_m<k^2,
     \end{cases}
\end{equation}
and 
\begin{equation}\label{TTH2}
   \overline{\beta_{h,k,\varepsilon}\varphi_{1}^3}-\overline{\varphi_1}\Big(\overline{\beta_{k,h,\varepsilon}\varphi_{1}^2}\Big)\begin{cases}
     >0 & \text{if}\quad d_I\lambda_m<2k^2,\cr 
     <0 & \text{if}\quad d_I\lambda_m>2k^2.
   \end{cases}
\end{equation}
Therefore,
\begin{enumerate}
    \item[\rm (i)] if $k^2<d_I\lambda_m<2k^2$, then $ 1>\overline{({\gamma_{k,h,\varepsilon}}/{\beta_{k,h,\varepsilon}})}\mathcal{R}_1$ and $\overline{\beta_{h,k,\varepsilon}\varphi_{1}^3}>\overline{\varphi_1}\Big(\overline{\beta_{k,h,\varepsilon}\varphi_{1}^2}\Big) $  for $0<\varepsilon\ll 1$.
    \item[\rm (ii)] if $2k^2<d_I\lambda_m$, then $\overline{\beta_{h,k,\varepsilon}\varphi_{1}^3}<\overline{\varphi_1}\Big(\overline{\beta_{k,h,\varepsilon}\varphi_{1}^2}\Big)$  for $0<\varepsilon\ll 1$.
\end{enumerate}
    
\end{prop}
\begin{proof}
    It is easy to see from \eqref{R-star-pde} and \eqref{prop-eq2} that $l^*(\varepsilon)=1/{\mathcal{R}_1}$.  Now, by Proposition, $l^*(\varepsilon)$ can be written as 
    \begin{equation}\label{PLP1}
       1/{\mathcal{R}_1} =l^*(\varepsilon)=k+{\varepsilon}\overline{h}+\varepsilon^2l_2^*+\varepsilon^3l_2^*(\varepsilon)\quad 0<\varepsilon\ll 1,
    \end{equation}
    where $l_2^*$ is given by \eqref{l-star-2-eq} and $l_3(\varepsilon)$ satisfies \eqref{prop-eq4}.  As a result, since 
    $ \overline{{\gamma_{k,h,\varepsilon}}/{\beta_{k,h,\varepsilon}}}=\overline{\beta_{k,h,\varepsilon}}=\overline{(k+\varepsilon h)}=k+\varepsilon\overline{h}$, 
    we can then employ  proposition \eqref{appen-prop2} to get
    $$
    \Big({1}/{\mathcal{R}_1}-\overline{{\gamma_{k,h,\varepsilon}}/{\beta_{k,h,\varepsilon}}}\Big)/\varepsilon^2=l_2^*-\varepsilon l_{3}^*(\varepsilon)=\Big(\Big(1-{k^2}/({d_I\lambda_m})\Big)\overline{h^2}\Big)/k-\varepsilon l^*_3(\varepsilon)\quad 0<\varepsilon\ll 1.
    $$
    Thus \eqref{TTH1} holds since $\varepsilon l^*_3(\varepsilon)\to 0 $ as $\varepsilon\to 0$ by \eqref{prop-eq4}. Next, we prove that \eqref{TTH2} holds.

\quad    To this end, set $\tilde{\varphi}=\varphi_{1}/(\int_{\Omega}\varphi_{1})$. Then 

\begin{equation}\label{TTH3}
\overline{\beta_{h,k,\varepsilon}\varphi_1^3}-\overline{\varphi_{1}}\Big(\overline{\beta_{k,h,\varepsilon}\varphi_{1}^2}\Big)=\big(\overline{\beta_{k,h,\varepsilon}\tilde{\varphi}^3}-\overline{\beta_{k,h,\varepsilon}\tilde{\varphi}^2}\big)\|\varphi_{1}\|_{L^1(\Omega)}^3.
\end{equation}
Now, observe from \eqref{R-star-pde}, and the fact $l^*(\varepsilon)=1/\mathcal{R}_1$ and $\int_{\Omega}\tilde{\varphi}=1$, that $\tilde{\varphi}$ is the unique solution of \eqref{prop-eq2}.  Then, by proposition \ref{appen-prop1}, $\tilde{\varphi}$ can be written as
\begin{equation}\label{TTH4}
    \tilde{\varphi}=1/{|\Omega|}+\varepsilon\tilde{\varphi}_1+\varepsilon^2\tilde{\varphi}_2+\varepsilon^3\tilde{\varphi}_3(\cdot;\varepsilon)\quad 0<\varepsilon\ll 1,
\end{equation}
where $\tilde{\varphi}_{3}$ satisfies \eqref{prop-eq4}. Hence 
\begin{align}\label{TTH7}
\overline{\beta_{k,h,\varepsilon}\tilde{\varphi}^3}-\overline{\beta_{k,h,\varepsilon}\tilde{\varphi}^2}
=\varepsilon\int_{\Omega}\beta_{k,h,\varepsilon}\tilde{\varphi}^2(\tilde{\varphi}_1+\varepsilon\tilde{\varphi}_2+\varepsilon^2\tilde{\varphi}_3)
\end{align}
For convenience, we set $P=\tilde{\varphi}_2+\varepsilon\tilde{\varphi}_3$ and $Q=\tilde{\varphi}_1+\varepsilon P$. So, $\tilde{\varphi}={1}/{|\Omega|}+\varepsilon Q$ and 
\begin{align*}
    \beta_{k,h,\varepsilon}\tilde{\varphi}^2Q
    =& \frac{k}{|\Omega|^2}\tilde{\varphi}_1+\varepsilon\Big(\frac{2 k }{|\Omega|}\tilde{\varphi}_1^2+k \tilde{\varphi}^2 P +h \tilde{\varphi}^2Q\Big)+k\varepsilon^2\Big(\frac{2}{|\Omega|}P +Q^2\Big)\tilde{\varphi}_1 
\end{align*}
Observing that
\begin{align*}
    k \tilde{\varphi}^2 P
    =\frac{k}{|\Omega|^2}\tilde{\varphi}_2+k\varepsilon\Big(\frac{\tilde{\varphi}_3}{|\Omega|^2}+\Big(\frac{2}{|\Omega|}+\varepsilon Q\Big)QP\Big)\ \text{and}\
    h \tilde{\varphi}^2 Q
    =\frac{1}{|\Omega|^2}h\tilde{\varphi}_1+\varepsilon h\Big(\frac{P}{|\Omega|^2}+\Big(\frac{2}{|\Omega|}+\varepsilon Q\Big)Q^2\Big)
    \end{align*}
then
\begin{align*}
    \beta_{k,h,\varepsilon}\tilde{\varphi}^2Q
    =\frac{k}{|\Omega|^2}\tilde{\varphi}_1+\frac{\varepsilon}{|\Omega|^2}\Big({2k}{|\Omega|}\tilde{\varphi}_1^2+k\tilde{\varphi}_2+h\tilde{\varphi}_1\Big)+\varepsilon^2\tilde{\mathbb{H}}(\cdot;\varepsilon),
\end{align*}
where $$\tilde{\mathbb{H}}(\cdot;\varepsilon):=k\Big(\frac{2}{|\Omega|}P +Q^2\Big)\tilde{\varphi}_1+k\Big(\frac{\tilde{\varphi}_3}{|\Omega|^2}+\Big(\frac{2}{|\Omega|}+\varepsilon Q\Big)QP\Big)+h\Big(\frac{P}{|\Omega|^2}+\Big(\frac{2}{|\Omega|}+\varepsilon Q\Big)Q^2\Big).$$
As a result, since $\int_{\Omega}\tilde{\varphi}_i=0$, $i=1,2$, we get
\begin{equation}\label{TTH6}
    \frac{1}{\varepsilon}\int_{\Omega}\beta_{k,h,\varepsilon}\tilde{\varphi}^2Q=\frac{1}{|\Omega|^2}\Big(2k|\Omega|\int_{\Omega}\tilde{\varphi}_1^2+\int_{\Omega}h\tilde{\varphi}_1\Big)+\varepsilon\int_{\Omega}\tilde{\mathbb{H}}(\cdot;\varepsilon)\quad 0<\varepsilon\ll 1.
\end{equation}
Now, thanks to proposition \ref{appen-prop2}, $ 
2k|\Omega|\int_{\Omega}\tilde{\varphi}_1^2+\int_{\Omega}h\tilde{\varphi}_1=\Big(\frac{2k^2}{d_I\lambda_m}-1\Big)\frac{k}{|\Omega|d_I\lambda_m}\int_{\Omega}h^2.$  Since by \eqref{prop-eq4}, $\varepsilon\|\tilde{\mathbb{H}}(\cdot;\varepsilon)\|_{\infty}\to 0$ as $\varepsilon\to 0$, we get from \eqref{TTH3}, \eqref{TTH7}, and \eqref{TTH6}  that \eqref{TTH2} holds.

\end{proof}


\begin{thebibliography}{9}



\bibitem{Allen2008}L.J.S. Allen, B.M. Bolker, Y. Lou, A.L. Nevai, Asymptotic profiles of the steady states for an SIS epidemic reaction-diffusion model, {\it Disc. Cont. Dyn. Syst.} {\bf 21} (2008), 1-20.




\bibitem{CastellanoSalako2021} K. Castellano, R. B. Salako, On the effect of lowering population's movement to control the spread of infectious disease  {\it J. Diff. Equat.}, {\bf316} (2022), 1-27. 


\bibitem{Cui_Lou2016}R. Cui, Y. Lou, A spatial SIS model in advective heterogeneous environments, {\it J. Diff. Equat.}, {\bf 261}
(2016), 3305-3343.






\bibitem{DengWu2016}K. Deng, Y. Wu, Dynamics of a susceptible-infected-susceptible epidemic reaction-diffusion model, {\it  Proc. Roy.
Soc. Edinburgh Sect. A}, {\bf 146} (2016), 929-946.

\bibitem{Evans} L. C. Evans, Partial Differential Equations: Second Edition (Graduate Studies in Mathematics) 2nd Edition

\bibitem{GKLZ2015}J. Ge, K.I. Kim, Z. Lin, H. Zhu, A SIS reaction-diffusion-advection model in a low-risk and high-risk domain,
{\it J. Diff. Equat.}, {\bf259} (2015), 5486-5509.

\bibitem{Dan_Henry} D. Henry, Geometric Theory of Semilinear Parabolic Equations. Lectures Notes in Mathematics 840, Springer-Verlag, New York, 1981




\bibitem{DeJong1995} M. C. M. de Jong, et al., How does transmission of infection depend on population size?, in Epidemic Models: Their Structure and Relation to Data, {\it Cambridge University Press}, 1995, pp.84–89.

\bibitem{LP2022} H. Li, R. Peng, An SIS epidemic model with mass action infection mechanism in a patchy environment, Studies in Applied Mathematics, {\bf 150} (2022), 650 - 704.



\bibitem{LSS2023}  Y. Lou, R. B. Salako, P. Song, Human Mobility and Disease Prevalence, J.  Math. Biol., {\bf 87}, no. 1, (2023), 1-32.

\bibitem{LS2023_2}  Y. Lou, R. B. Salako, Mathematical analysis on the coexistence of strains  in some reaction-diffusion systems, Journal of Differential Equations, {\bf 370}, (2023), 424-469.

\bibitem{LouSalako2021} Y. Lou, R. B. Salako, Control Strategy for multiple strains epidemic model, {\it Bulletin of Mathematical Biology}, {\bf 84} 10 (2022), p.1-47.




\bibitem{Peng2009a} R. Peng, Asymptotic profiles of the positive steady state for an SIS epidemic reaction-diffusion model. Part I, {\it J. Diff. Equat.}, {\bf 247} (2009), 1096-1119.

\bibitem{Peng2009b} R. Peng, S. Liu, Global stability of the steady states of an SIS epidemic reaction-diffusion model, {\it Nonlinear Anal.} {\bf 71} (2009), 239-247.

\bibitem{Peng_Shi2008} R. Peng, J. Shi, M. Wang, On stationary patterns of a reaction-diffusion model with autocatalysis and saturation law, {\it Nonlinearity}, 21 (2008), 1471-1488.

\bibitem{Peng_Yi2013} R. Peng, F. Yi, Asymptotic profile of the positive steady state for an SIS epidemic reaction-diffusion model:
Effects of epidemic risk and population movement, {\it Phys. D}, {\bf 259} (2013), 8-25.

\bibitem{Peng_Zhao} R. Peng, X. Zhao, A reaction-diffusion SIS epidemic model in a time-periodic environment, Nonlinearity, 25 (2012), 1451-1471.

\bibitem{Salako2023_1} R. B. Salako, Impact of environmental heterogeneity,  population size and movement on the persistence of a two-strain infectious disease, Journal of Mathematical Biology {\bf 86},  1 (2023), 1-36.

\bibitem{TW2023}Y. Tao, M. Winkler, Analysis of a chemotaxis-SIS epidemic model with unbounded infection force, Nonlinear Analysis: Real World Applications, {\bf  71} (2023), 103820

%
\bibitem{Wu_Zou2016} Y. Wu and Z. Zou, Asymptotic profiles of steady states for a diffusive SIS epidemic model with mass action infection mechanism, {\it J. Diff. Equat.} {\bf 261}(2016) 4424–4447. 

\bibitem{Wen2018}X. Wen, J. Ji, B. Li, Asymptotic profiles of the endemic equilibrium to a diffusive SIS epidemic model with mass action infection mechanism, {\it J. Math. Anal. Appl.} {\bf 458} (2018), 715-729.




\end{thebibliography}
\end{document}